\documentclass[a4paper]{article}

\usepackage{graphicx}

\usepackage[utf8]{inputenc} 
\usepackage[T1]{fontenc}     
 \usepackage[english]{babel}      
\usepackage{geometry}          
\usepackage{amsmath, amstext, ,amscd, amssymb, mathrsfs, amsthm}
\usepackage{stmaryrd}
\usepackage{mathrsfs}
\usepackage{hyperref}
\usepackage{enumitem}
\usepackage[all]{xy}

\newcommand{\R}{\mathbb{R}}
\newcommand{\C}{\mathbb{C}}

\newcommand{\Z}{\mathbb{Z}}
\newcommand{\N}{\mathbb{N}}

\newcommand{\F}{\mathrm{F}}
\newcommand{\G}{\mathrm{G}}

\newcommand{\smooth}{\mathscr{C}^{\infty}}

\newcommand{\ic}{\sqrt{-1}}

\newcommand{\bw}{\overline{w}}
\newcommand{\M}{\mathscr{M}_M}

\newcommand{\derpar}[2]{\frac{\partial#1}{\partial#2}}

\newcommand{\e}{\varepsilon}

\newcommand{\E}{\mathbb{E}}

\newcommand{\db}{\bar{\partial}}
\newcommand{\Wedge}{\Lambda^{0,\bullet}}

\newcommand{\End}{\mathrm{End}}

\newcommand{\LL}{\mathscr{L}}
\newcommand{\0}{\mathscr{O}}

\newcommand{\n}{\nabla}
\newcommand{\ol}{\overline}

\newcommand{\wt}[1]{\widetilde{#1}}

\DeclareMathOperator{\Ker}{Ker}

\DeclareMathOperator{\rank}{rk}

\DeclareMathOperator{\Id}{Id} 

\DeclareMathOperator{\tr}{Tr}

\renewcommand{\Im}{\mathrm{Im}}

\newtheorem{theorem}{Theorem}
\newtheorem{lemma}{Lemma}

\theoremstyle{definition}
\newtheorem{definition}{Definition}

\theoremstyle{remark}
\newtheorem{remark}{Remark}
\newtheorem{example}{Example}
   
\begin{document}

\title{Holomorphic Morse inequalities for orbifolds}


\author{Martin Puchol\thanks{ This work was performed within the framework of the LABEX MILYON (ANR-10-LABX-0070) of Universit\'e de Lyon, within the program "Investissements d'Avenir" (ANR-11-IDEX-0007) operated by the French National Research Agency (ANR).}}


\maketitle

\begin{abstract}
We prove that Demailly's holomorphic Morse inequalities hold true for complex orbifolds by using a heat kernel method. Then we introduce the class of Moishezon orbifolds and as an application of our inequalties, we give a geometric criterion for a compact connected orbifold to be a Moishezon orbifolds, thus generalizing Siu's and Demailly's answers to the Grauert-Riemenschneider conjecture to the orbifold case.
\end{abstract}

\section{Introduction}
\label{intro}

Morse Theory investigates the topological information carried by Morse functions on a manifold and in particular their critical points. Let $f$ be a Morse function on a compact manifold of real dimension $n$.  Let $m_j$ $(0 \leq j \leq n)$ be the the number of critical points of $f$ of Morse index  $j$, and let $b_j$ be the Betti numbers of the manifold. Then the strong Morse inequalities state that for $0\leq q \leq n$,
\begin{equation}
\label{MI-reelles}
\sum_{j=0}^q(-1)^{q-j} b_j \leq \sum_{j=0}^q(-1)^{q-j} m_j,
\end{equation}
with equality if $q=n$. From \eqref{MI-reelles}, we get the weak Morse inequalities:
\begin{equation}
b_j\leq m_j \qquad \text{for} \quad 0\leq j\leq n.
\end{equation}

In his seminal paper \cite{MR683171}, Witten  gave an analytic proof of the Morse inequalities by analyzing the spectrum of the Schr\"{o}dinger operator $\Delta_t=\Delta +t^2|df|^2+tV$, where $t>0$ is a real parameter and $V$ an operator of order 0. For $t\to +\infty$, Witten shows that the spectrum of $\Delta_t$ approaches in some sense the spectrum of a sum of harmonic oscillators attached to the critical points of $f$.  

In \cite{demailly}, Demailly established analogous asymptotic Morse inequalities for the Dolbeault cohomology associated with high tensor powers $L^p:=L^{\otimes p}$ of a (smooth) holomorphic Hermitian line bundle $(L,h^L)$ over a (smooth) compact complex manifold $(M,J)$ of dimension $n$.  The inequalities of Demailly give asymptotic bounds on the Morse sums of the Betti numbers of $\db$ on $L^p$ in terms of certain integrals of the Chern curvature $R^L$ of $(L,h^L)$. More precisely,  we define $\dot{R}^L\in \End(T^{(1,0)}M)$ by $g^{TM}(\dot{R}^Lu,\ol{v}) =R^L(u,\ol{v})$ for $u,v\in T^{(1,0)}M$, where $g^{TM}$ is any $J$-invariant Riemannian metric on $TM$. We denote by $M( q)$ the set of points where $\dot{R}^L$ is non-degenerate and exactly $q$ negative eigenvalues, and we set $M(\leq q) = \cup_{j\leq q} M(j)$. Let $n=\dim_\C M$, then we have for $0\leq q \leq n$
\begin{equation}
\label{IM}
\sum_{j=0}^q (-1)^{q-j} \dim H^j(M, L^p) \leq \frac{p^n}{n!} \int_{M(\leq q)} (-1)^q \left( \frac{\ic}{2\pi}R^{L}\right)^{n} +o(p^n),
\end{equation}
with equality if $q=n$. Here $H^j(M,L^p)$ denotes the Dolbeault cohomology in bidegree $(0,j)$, which is also the $j$-th group of cohomology of the sheaf of holomorphic sections of $L^p$. Note that $M(q)$ and $M(\leq q)$ are open subsets of $M$ and do not depend on the metric $g^{TM}$.

These inequalities have found numerous applications. In particular, Demailly used them in \cite{demailly} to find new geometric characterizations of Moishezon spaces, which improve Siu's solution in \cite{MR755233,MR797421} of the Grauert-Riemenschneider conjecture \cite{MR0302938}. Another notable application of the holomorphic Morse inequalities is the proof of the effective Matsusaka theorem by Siu \cite{MR1275204,MR1603616}. Recently, Demailly used these inequalities in \cite{MR2918158} to prove a significant step of a generalized version of the Green-Griffiths-Lang conjecture.

To prove these inequalities, the key remark of Demailly was that in the formula for the Kodaira Laplacian $\square_p$ associated with $L^p$, the metric of $L$ plays formally the role of the Morse function in the paper Witten \cite{MR683171}, and that the parameter $p$ plays the role of the parameter $t$. Then the Hessian of the Morse function becomes the curvature of the bundle. The proof of Demailly was based on the study of the semi-classical behavior as $p\to +\infty$ of the spectral counting functions of $\square_p$. Subsequently, Bismut gave an other proof of  the holomorphic Morse inequalities in \cite{MR886814} by adapting his heat kernel proof of the Morse inequalities \cite{MR852155}. The key point is that we can compare the left hand side of \eqref{IM} with the alternate trace of the heat kernel acting on forms of degree $\leq q$ (see \cite[Theorem 1.4]{MR886814}) :
\begin{equation}
\label{Morse-sum-and-supertrace}
\sum_{j=0}^q (-1)^{q-j} \dim H^j(M, L^p) \leq \sum_{j=0}^q (-1)^{q-j} \tr\left[ e^{-\frac{u}{p}\square_p}|_{\Omega^{0,j}(M,L^p)}\right],
\end{equation}
with equality if $q=n$. Then, Bismut obtained the holomorphic Morse inequalities by showing the convergence of the heat kernel thanks to probability theory. Demailly \cite {MR1128538} and Bouche \cite{MR1056777} gave an analytic approach of this result. In \cite{ma-marinescu}, Ma and Marinescu  gave a new proof of this convergence, replacing the probabilistic arguments of Bismut  \cite{MR886814} by arguments inspired by the analytic localization techniques of Bismut-Lebeau \cite[Chapter 11]{bismut-lebeau}.

When the bundle $L$ is positive, \eqref{IM} is a consequence of the Hirzebruch-Riemann-Roch theorem and of the Kodaira vanishing theorem, and  reduces to
\begin{equation}
\dim H^0(M, L^p)=\frac{p^n}{n!} \int_{M}  \left(\frac{\ic}{2\pi} R^L\right)^{n} +o(p^n).
\end{equation}
 In this case, a local estimate can be obtained by the study of the asymptotic of the Bergman kernel (the kernel of the orthogonal projection from $\smooth(M,L^p)$ onto $H^0(M,L^p)$) when $p\to+\infty$.  We refer to \cite{ma-marinescu} and the reference therein for the study of the Bergman kernel. 

It is a natural question to know whether we can prove a version of Demailly's holomorphic Morse inequalities when $M$ is a complex orbifolds and $L$ is an orbifold bundles. Applying the result of \cite{MR3545500,MR2917156}, one can prove that such inequalities hold if $M$ is the quotient of a CR manifold by a transversal CR $\mathbb{S}^1$-action. In this paper, we prove that Demailly's  inequalities hold for high tensor power of an orbifold line bundle, twisted by another orbifold bundle, on a general compact complex orbifolds. We also introduce the class of Moishezon orbifolds and as an application of our inequalties, we give a geometric criterion for a compact connected orbifold to be a Moishezon orbifold, thus generalizing the above-mentionned Siu's and Demailly's answers to the Grauert-Riemenschneider conjecture to the orbifold case. \\

We now give more details about our results.

Let $M$ be a compact complex orbifold of dimension $n$. We denote the complex structure of $M$ by $J$. Let $(L,h^L)$ a Hermitian holomorphic orbifold line bundle on $M$ and let $(E,h^E)$ be a Hermitian holomorphic orbifold vector bundle on $M$. As we will see in Section \ref{sect-complex-orbifolds}, we may assume without loss of generality that $L$ and $E$ are proper. We denote by $R^L$ the Chern curvature of $(L,h^L)$. We refer to Section \ref{Sect-Orbifold} for the background concerning orbifold, but for this introduction, let us just say that every object on orbifold can be seen locally as being the quotient of an object on a non-singular manifold which is invariant by a finite group, and that we keep the same notation for both these objects.

We define $\dot{R}^L$, $M(q)$ and $M(\leq q)$ exactly as in the non-singular case above.

Let $H^\bullet(M,L^p \otimes E)$ be the orbifold Dolbeault cohomology. Then our holomorphic Morse inequalities for orbifolds have the same statement as the regular ones:

\begin{theorem}
\label{OHMI}
As $p\to +\infty$, the following strong Morse inequalities hold for $q\in \{ 1,\dots, n\}$
 \begin{equation}
 \label{Eq-OHMI}
\sum_{j=0}^q (-1)^{q-j}\dim H^j(M,L^p\otimes E) \leq \rank(E) \frac{p^{n}}{n!}  \int_{M(\leq q)} (-1)^{q} \left(\frac{\ic}{2\pi} R^L\right)^{n} + o(p^{n}),
\end{equation}
with equality for $q=n$.

In particular, we get the weak Morse inequalities
 \begin{equation}
\dim H^q(M,L^p\otimes E)
\leq \rank(E) \frac{p^{n}}{n!}  \int_{M( q)} (-1)^{q} \left(\frac{\ic}{2\pi} R^L\right)^{n}+ o(p^{n}).
\end{equation}

The integrals over an orbifold appearing here are defined by \eqref{defn-integration-eq}.
\end{theorem}

Of course, if the orbifold $M$ is a non-singular manifold, and if the orbifold bundles $L$ and $E$ are usual bundles, then Theorem \ref{OHMI} coincides with the the inequalities of Demailly. \\

Let us now give the main steps our our proof. We draw our inspiration from the heat kernel method of \cite{MR886814} (see also \cite[Sections. 1.6 and 1.7]{ma-marinescu}), but the main difficulty when compared to this paper is that the singularities of $M$ make it impossible to have an uniform asymptotics for the heat kernel (see Remark \ref{rem-dvpt-non-unif}), and we thus have to work further near the singularities. In particular we have to use the off-diagonal development of the heat kernel proved by Dai-Liu-Ma \cite{MR2215454}. 

Note that in the case where $L$ is positive, similar difficulties arise in the study of the Bergman kernel, and the asymptotic results concerning the heat kernel given below have parallel results for the Bergman kernel on orbifolds \cite{ma-marinescu} (see also \cite{MR2215454}).

Let $g^{TM}$ be a Riemannian metric on $TM$ which is compatible with $J$. Let 
\begin{equation}
\square_p:=\db^{L^p\otimes E}\db^{L^p\otimes E,*}+\db^{L^p\otimes E,*}\db^{L^p\otimes E}
\end{equation}
 be the Kodaira Lpalacian acting on  $\Omega^{0,\bullet}(M,L^p\otimes E)$ associated with $L^p\otimes E$, $h^L$, $h^E$ and $g^{TM}$ (see the beginning of Section \ref{cvce-noyau-Sect} for details). Note that the operator $\square_p$  preserves the $\Z$-grading. We denote by $\tr_q[e^{-\frac{u}{p}\square_p}]$ the trace of $e^{-\frac{u}{p}\square_p}$ acting on $\Omega^{0,q}(M,L^p\otimes E)$. We then have an analogue of \eqref{Morse-sum-and-supertrace}:
\begin{theorem}
\label{Euler-et-chaleur}
 For any $u>0$, $p\in \N^*$ and $0\leq q \leq n$, we have
 \begin{equation}
 \label{Euler-et-chaleur-eq}
\sum_{j=0}^q (-1)^{q-j} \dim H^j(M,L^p\otimes E) \leq \sum_{j=0}^q (-1)^{q-j} \tr_j\big[e^{-\frac{u}{p}\square_p}\big],
\end{equation}
with equality for $q=n$.
\end{theorem}
  
From this result, we will prove our inequalities \eqref{Eq-OHMI} by studying the asymptotics of the heat kernel. We begin by its behavior away from the singularities of $M$.
  
Let $\{w_j\}_{j=1}^n$ be a local  orthonormal frame of $T^{(1,0)}M$ (with the metric induced by $g^{TM}$) with dual frame $\{w^j\}_{j=1}^n$. Set
 \begin{equation}
\omega_d = -\sum_{k,\ell} R^{L} (w_k,\bw_\ell) \bw^\ell\wedge i_{\bw_k} \in \End(\Wedge (T^*M)).
\end{equation}
  
  For compactness, in the sequel we will write
  \begin{equation}
  \label{def-Lim-Eq}
\mathcal{L}im_u(x)= \frac{1}{(2\pi)^{n}}\frac{\det(\dot{R}_{x}^{L})e^{u\omega_{d,x}}}{\det\big(1-\exp(-u\dot{R}_{x}^{L})\big)}\otimes \Id_{E_x},
\end{equation}
with the convention that if an eigenvalue of $\dot{R}^{L}_{x}$ is zero, then its contribution to the term $\frac{\det(\dot{R}_{x}^{L})}{\det\big(1-\exp(-u\dot{R}_{x}^{L})\big)}$ is $\frac{1}{u}$. 

Let $M_{reg}$ be the regular part of $M$ (see Section \ref{Sect-Orbifold}) and let $M_{sing}=M\setminus M_{reg}$ be the singular part of $M$.
\begin{theorem}
\label{asymp-away-sing-Thm}
 For $K\subset M_{reg}$ compact, $u>0$ and $\ell\in \N$, there exists $C>0$ such that for any $x\in K$, we have as $p\to+\infty$
 \begin{equation}
\left|p^{-n}e^{-\frac{u}{p}\square_p}(x,x)-\mathcal{L}im_u(x)\right|_{\mathscr{C}^\ell}\leq Cp^{-1/2}.
\end{equation}
Here, $|\cdot|_{\mathscr{C}^\ell}$ denotes the $\mathscr{C}^\ell$-norm.
\end{theorem}

\begin{remark}
 In fact, the bound in the right hand side in Theorem \ref{asymp-away-sing-Thm} can be improved to $p^{-1}$ using the same reasoning as in \cite[Section 4.2.4]{ma-marinescu} (see also \cite{MR2215454}). However, we will not need this improvement and we leave it to the reader.
\end{remark}

We now turn to the asymptotic behavior of the heat kernel near the singularities. This is the main technical innovation of this paper.

Let $\n^L$ and $\n^E$ be the Chern connections of $(L,h^L)$ and $(E,h^E)$, i.e., the unique connections preserving both the holomorphic and Hermitian structures.

 Let $x_0\in M_{sing}$, from Section \ref{Sect-Orbifold} (and in particular Lemma \ref{Lem-carte-lineaire}), we know that we can identify an open neighborhood $U_{x_0}$ of $x_0$ to $\wt{U}_{x_0}/G_{x_0}$, where $\wt{U}_{x_0} \subset \C^n$ is an open neighborhood of 0 on which the finite group $G_{x_0}$ acts linearly and effectively. In this chart, $x_0$ correspond to the class $[0]$. Then the metric $g^{TM}$ induces a $G_{x_0}$-invariant metric on $\wt{U}_{x_0}$.
 
 Let $\wt{U}_{x_0}^g$ be the fixed point-set of $g\in G_{x_0}$, and let $\wt{N}_{x_0,g}$ be the normal bundle of $\wt{U}_{x_0}^g$ in $\wt{U}_{x_0}$. For each $g\in G_{x_0}$, the exponential map $Y\in (\wt{N}_{x_0,g})_{\tilde{x}} \mapsto \exp_{\tilde{x}}^{\wt{U}_{x_0}}(Y)$ identifies a neighborhood of $\wt{U}_{x_0}^g$ to $\wt{W}_{x_0,g}=\{Y\in \wt{N}_{x_0,g} \: : \: |Y| \leq \e \}$. We identify $L|_{\wt{W}_{x_0,g}}$ and $E|_{\wt{W}_{x_0,g}}$ to $L|_{\wt{U}_{x_0}^g}$ and $E|_{\wt{U}_{x_0}^g}$ by using the parallel transport (with respect to $\n^L$ and $\n^E$) along the above exponential map. Then the action of $g$ on $L|_{\wt{W}_{x_0,g}}$ is the multiplication by $e^{i\theta_g}$, and $\theta_g$ is locally constant on $\wt{U}_{x_0}^g$. Likewise, the action of $g$ on $E|_{\wt{W}_{x_0,g}}$ is given by $g^E\in \smooth(\wt{U}_{x_0}^g,\End(E))$, and $g^E$ is parallel with respect to $\n^E$. 
 
 If $\wt{Z}\in \wt{W}_{x_0,g}$, we write $\wt{Z}=(\wt{Z}_{1,g},\wt{Z}_{2,g})$ with $\wt{Z}_{1,g} \in \wt{U}^g_{x_0}$ and $\wt{Z}_{2,g}\in \wt{N}_{x_0,g}$.
 
On $\wt{U}_{x_0}$, we have two metrics: the first is the $G$-invariant lift $\wt{g^{TM}}$ of $g^{TM}|_{U_{x_0}}$ and the second is the constant metric $(\wt{g^{TM}})_{\wt{Z}=0}$. Let $dv_{\wt{M}}$ and $dv_{\wt{TM}}$ be the associated volume form, and let $\wt{\kappa}$ be the smooth positive function defined by
\begin{equation}
\label{def-kappa-intro}
dv_{\wt{M}}(\wt{Z}) = \wt{\kappa}(\wt{Z}) dv_{\wt{TM}}(\wt{Z}),
\end{equation}
with $\wt{\kappa}(0)=1$.
 
 For $x\in U_{x_0}$, the $G$-invariant lift of $\dot{R}^{L}_{x}$ acts on $T^{1,0}\wt{U}_{x_0}$, and we extend it to $T\wt{U}_{x_0}\otimes \C = T^{1,0}\wt{U}_{x_0}\oplus T^{0,1}\wt{U}_{x_0}$ by setting $\dot{R}^{L}_{x}\ol{\wt{v}}=-\ol{\dot{R}^{L}_{x} \wt{v}}$. We then define
 \begin{equation}
 \label{def-E-Eq}
\mathcal{E}_{g,x}(u,\wt{Z})=
\exp\Bigg\{-\bigg\langle\frac{\dot{R}^L_{x}/2}{\mathrm{th}(u\dot{R}^L_{x}/2)} \wt{Z},\wt{Z} \bigg\rangle+\bigg\langle\frac{\dot{R}_{x}^{L}/2}{\mathrm{sh}(u\dot{R}_{x}^{L}/2)} e^{u\dot{R}_{x}^{L}/2} g^{-1}\wt{Z},\wt{Z} \bigg\rangle\Bigg\}.
\end{equation}
Here again, we use the convention that if an eigenvalue of $\dot{R}^{L}_{x}$ is zero, then the contribution of the associated eigenspace to $\mathcal{E}_{g,x}(u,\wt{Z})$ is of the form $(u,\wt{V})\mapsto e^{-\frac{1}{2u}|g^{-1}\wt{V}-\wt{V}|^2}$.

\begin{theorem}
\label{asymp-near-sing-Thm}
 On $\wt{U}_{x_0}$, for $u>0$ and $\ell\in \N$, there exist $c,C>0$ and $N\in \N$ such that for any $|\wt{Z}|<\e/2$, as $p\to+\infty$
\begin{multline}
\label{asymp-near-sing-Eq}
 \Bigg|p^{-n}e^{-\frac{u}{p}\square_p}(\wt{Z},\wt{Z})-\mathcal{L}im_u(\wt{Z})\\
 -\sum_{\substack{g\in G_{x_0} \\ g\neq 1}}e^{ip\theta_g}g^{E}(\wt{Z}_{1,g})\kappa^{-1}_{\wt{Z}_{1,g}}(\wt{Z}_{2,g}) \mathcal{L}im_u(\wt{Z}_{1,g})\mathcal{E}_{g,\wt{Z}_{1,g}}(u,\sqrt{p}\wt{Z}_{2,g})\Bigg|_{\mathscr{C}^\ell}\\
 \leq Cp^{-1/2}+ Cp^{\frac{\ell-1}{2}}(1+\sqrt{p}\,d(Z,M_{sing}))^Ne^{-cp\,d(Z,M_{sing})^2}.
\end{multline}
\end{theorem}
  
\begin{remark}
The term $\mathcal{L}im_u(\wt{Z}_{1,g})\mathcal{E}_{g,\wt{Z}_{1,g}}(u,\sqrt{p}\wt{Z}_{2,g})$ appearing in \eqref{asymp-near-sing-Eq} can be seen as the heat kernel at the time $u$ of some explicit harmonic oscillator (depending on $\wt{Z}_{1,g}$) on $\R^{2n}$, evaluated at $(\sqrt{p}g^{-1}\wt{Z}_{2,g},\sqrt{p}\wt{Z}_{2,g})$. For more details see \eqref{chaleur-L} and \eqref{cvce-noyau-square_p-Ep}.
\end{remark}

\begin{remark}
\label{rem-dvpt-non-unif}
 From Theorem \ref{asymp-near-sing-Thm}, for $x\in M_{sing}$, we have $ \big|p^{-n}e^{-\frac{u}{p}\square_p}(x,x)-|G_{x}|\mathcal{L}im_u(x)\big|\leq Cp^{-1/2}$. In particular, unlike in the usual non-singular case, if $M_{sing}$ is not empty, the asymptotics of Theorem \ref{asymp-away-sing-Thm} cannot be uniform on $M_{reg}$.
 
 An analogous result holds true for the Bergman kernel, see \cite{ma-marinescu} or \cite{MR2215454}.
\end{remark}

Next, we give an application of the inequalities \eqref{Eq-OHMI}.  The class of Moishezon orbifolds is defined in a similar way as usual regular Moishezon manifolds: we begin by defining meromorphic functions on an orbifold as local quotient of holomorphic orbifold functions, and we say that a compact connected orbifold $M$ is a \emph{Moishezon orbifold} if it possesses $\dim M$ meromorphic functions that are algebraically independent, i.e., they satisfy no non-trivial polynomial equation (see Section \ref{sect-def-Moishezon} for details).

For regular Moishezon manifold, Grauert and Riemenschneider \cite{MR0302938} conjectured that if a compact connected manifold $M$ possesses a smooth Hermitian bundle which is semi-positive everywhere and positive on an open dense set, then $M$ is Moishezon. Siu proved a stronger version of this conjecture in \cite{MR755233,MR797421}, and Demailly improve Siu's result in \cite{demailly}. Here, we prove that these results are still valid for orbifolds.

\begin{theorem}
\label{thm-criterion}
  Let $M$ be a compact connected complex orbifold of dimension $n$ and let $(L,h^L)$ be a holomorphic orbifold line bundle on $M$. If one of the following conditions holds:
\begin{enumerate}[label=(\roman*)]
\item (Siu-type criterion) $(L,h^L)$ is semi-positive and positive at one point,
\item (Demailly-type criterion) $(L,h^L)$ satisfies
  \begin{equation}
  \label{eq-criterion}
\int_{M(\leq 1)} \left(\frac{\ic}{2\pi} R^L\right)^n >0,
\end{equation}
\end{enumerate}
then $M$ is a Moishezon orbifold.
\end{theorem}

Note that in the course of the proof of Theorem \ref{thm-criterion}, we prove an orbifold version of the famous Siegel's Lemma \cite{MR0074061} (see also \cite[Lemma 2.2.6]{ma-marinescu}), see Theorem \ref{thm-Siegel-orbifold}.\\

This paper is organized as follows. In Section  \ref{Sect-Orbifold} we recall the definitions and basic properties of orbifolds. In Section \ref{cvce-noyau-Sect} we study the convergence of the heat kernel and prove Theorems \ref{asymp-away-sing-Thm} and \ref{asymp-near-sing-Thm}. In Section \ref{Sect-proof} we use these asymptotic result to prove the holomorphic Morse inequalities (Theorem \ref{OHMI}). Finally, in Section \ref{Sect-moishezon}, we use these inequalities to give a geometric criterion for a compact connected orbifold to be a Moishezon orbifold (Theorem \ref{thm-criterion}).

%%%%%%%%%%%%%%%%%%%%%%%%%%%%%%%%%%%%%%%%%%%%%%%%%%%%%%%%%%%%%%%%%%%%%%%%%%%%%%%%%%%%%%%%%%%%%%%%%%%%%%%%%%%%%%%%%%%%%%%%%%%%%%%%%%%%%%%%%%%%%%%%%%%%%%%%%
\section*{Acknowledgements}
The author thanks Xiaonan Ma and George Marinescu for helpful discussions and comments on the present paper..

%%%%%%%%%%%%%%%%%%%%%%%%%%%%%%%%%%%%%%%%%%%%%%%%%%%%%%%%%%%%%%%%%%%%%%%%%%%%%%%%%%%%%%%%%%%%%%%%%%%%%%%%%%%%%%%%%%%%%%%%%%%%%%%%%%%%%%%%%%%%%%%%%%%%%%%%%
\section{Background on orbifolds}
\label{Sect-Orbifold}

In this section we recall the background about orbifold. The content of this section is essentially taken from  \cite{ma-marinescu} and \cite{MR2883416}.
%%%%%%%%%%%%%%%%%%%%%%%%%%%%%%%%%%%%%%%%%%%%%%%%%%%%%%%%%%%%%%%%%%%%%%%%%%%%%%%%%%%%%%%%%%%%%%%%%%%%%%%%%%%%%%%%%%%%%%%%%%%%%%%%%%%%%%%%%%%%%%%%%%%%%%%%%
\subsection{Definitions}
\label{Sect-Orbifold-Def}

We first define a category $\mathcal{M}_s$ as follows.
\begin{description}
\item[Objects:] classes of pairs $(G, M)$ where $M$ is a connected smooth manifold and $G$ is a finite group acting effectively on $M$ (i.e., the unit is the unique element of $G$  acting as $\mathrm{Id}_M$);
\item[Morphisms:] a morphism $\Phi \colon (G, M) \to (G', M')$ is a family of open embeddings $\{\varphi \colon M\to M'\}_{\varphi\in \Phi}$ satisfying:
\begin{enumerate}
\item for each $\varphi\in \Phi$ there is an injective group morphism $\lambda_\varphi \colon G\hookrightarrow G'$ for which $\varphi$ is equivariant, i.e., $\varphi(g.x)=\lambda_\varphi(g).\varphi(x)$ for $x\in M$ and $g\in G$;
\item for $g'\in G'$ and $\varphi \in \Phi$, we have $g'.(\varphi(M))\cap\varphi(M)\neq \emptyset \Longrightarrow g'\in \lambda_\varphi(G)$;
\item for any $\varphi\in \Phi$, we have $\Phi=\{g'\varphi, \: g'\in G'\}$, where $g'\varphi \colon x\in M \mapsto g'.\varphi(x)\in M'$.
\end{enumerate}
\end{description}

\begin{definition}[Orbifold chart, atlas, structure]
 Let $M$ be a paracompact Hausdorff space.
 
 An $m$-dimensional \emph{orbifold chart} on $M$ consists of a connected open set $U$ of $M$, an object $(G_U,\widetilde{U})$ of $\mathcal{M}_s$ with $\dim \wt{U}=m$, and a ramified covering $\tau_U \colon \wt{U}\to U$ which is $G_U$-invariant and induces a homeomorphism $U\simeq \widetilde{U}/G_U$. We denote the chart by $(G_U,\widetilde{U})\overset{\tau_U}{\longrightarrow} U$.
 
 An $m$-dimensional \emph{orbifold atlas} $\mathcal{V}$ on $M$ consists of a family of $m$-dimensional orbifold charts $\mathcal{V}(U)=((G_U,\widetilde{U})\overset{\tau_U}{\longrightarrow} U)$ satisfying the following conditions:
\begin{enumerate}
 \item  the open sets $U\subset M$ form a covering $\mathcal{U}$ such that:
 \begin{equation}
 \label{eq:cond-recouvrement}
\text{for any } U,U'\in \mathcal{U} \text{ and } x\in U\cap U', \text{ there exists } U'' \in \mathcal{U} \text{ such that } x\in U'' \subset U\cap U'.
\end{equation}

\item  for any $U,V \in \mathcal{U}$ with $U \subset V$ there exists a morphism (of $\mathcal{M}_s$) $\Phi_{VU} \colon (G_U,\widetilde{U}) \to (G_V,\widetilde{V})$ which covers the inclusion $U\subset V$ and satisfies $\Phi_{WU}=\Phi_{WV}\circ \Phi_{VU}$ for any $U,V,W\in \mathcal{U}$ with $U \subset V \subset W$. These morphisms are called \emph{restriction morphisms}.
\end{enumerate}

It is easy to see that there exists a unique maximal orbifold atlas $\mathcal{V}_{max}$ containing $\mathcal{V}$: it consists of all orbifold charts $(G_U,\widetilde{U})\overset{\tau_U}{\longrightarrow} U$ which are locally isomorphic to charts from $\mathcal{V}$ in the neighborhood of each point of $M$. A maximal orbifold atlas $\mathcal{V}_{max}$ is called an \emph{orbifold structure} and the pair $(M,\mathcal{V}_{max})$ is called an \emph{orbifold}. As usual, once we have an orbifold atlas $\mathcal{V}$ on $M$, it uniquely determines a maximal atlas $\mathcal{V}_{max}$ and we denote the corresponding orbifold simply by $(M,\mathcal{V})$.
\end{definition}

In the above definition, we can replace $\mathcal{M}_s$ with a category of manifolds with an additional structure such as orientation, Riemannian metric, almost-complex structure or complex structure. In this case we require that the morphisms and the groups preserve the specified structure. In this way we can define oriented, Riemannian, almost-complex or complex orbifolds.

Certainly, for any object $(G_U,\widetilde{U})$ of $\mathcal{M}_s$, we can always construct a $G$-invariant Riemannian metric on $\widetilde{U}$. By a partition of unity argument, there always exists a Riemannian metric on an given orbifold $(M,\mathcal{V})$.

\begin{remark}
 Let $P$ be a smooth manifold, and let $H$ be a compact Lie group acting locally freely on $P$. Then the quotient space $P/H$ is an orbifold. Reciprocally, any orbifold $M$ can be presented by this way, with $H= O(n)$ ($n = \dim M$) see \cite[p. 76]{MR0474432} and \cite[p. 144]{MR641150}.
\end{remark}

\begin{definition}[Regular and singular set]
 Let $(M,\mathcal{V})$ be an orbifold. For each $x\in M$, we can choose a small neighborhood $(G_x,\widetilde{U}_x) \to U_x$ such that $x \in \wt{U}_x$ is a fixed point of $G_x$ (such $G_x$ is unique up to isomorphisms for each $x\in M$ from the definition). If the cardinal $|G_x|$ of $G_x$ is 1, then $x$ is a regular point of $M$, meaning that $M$ is a smooth manifold in a neighborhood of $x$. If $|G_x| > 1$, then $x$ is a singular point of $M$. We denote by $M_{sing} = \{x \in M\: :\: |G_x| > 1\}$ the singular set of $M$ and by $M_{reg}=M\setminus M_{sing}$ the regular set of $M$.
\end{definition}

The following lemma is proved in \cite[Lemma 5.4.3]{ma-marinescu}.
\begin{lemma}
\label{Lem-carte-lineaire}
 With the above notations, we can choose the local coordinates $\wt{U}_x \subset \R^m$ such that the finite growp $G_x$ acts linearly on $\R^m$ and $\{0\}=\tau_x^{-1}(x)$.
\end{lemma}

\textbf{In the sequel we will always use such charts.}

\begin{definition}[Orbifold vector bundle]
 An orbifold vector bundle $E$ over an orbifold $(M,\mathcal{V})$ is defined as follows: $E$ is an orbifold and for $U \in  \mathcal{U}$, $(G^E_U,\wt{p}_U : \wt{E}_U \to \wt{U})$ is a $G^E_U$-equivariant vector bundle where $(G^E_U,\wt{E}_U)$ gives the orbifold structure of $E$ and $(G_U =G^E_U/K_U^E,\wt{U})$, $K_U^E = \ker\big(G^E_U \to \mathrm{Diffeo}(\wt{U})\big)$,  gives the orbifold structure on $M$. We also require that for any $V\subset U$, each embedding in the restriction morphism of $E$ is an isomorphism of equivariant vector bundles compatible with an embedding in the restriction morphism of $X$.
 
  If moreover $G^E_U$ acts effectively on $\wt{U}$ for $U \in \mathcal{U}$, i.e. $K_U^E = \{1\}$, we call $E$ a proper orbifold vector bundle.
\end{definition}

Let $E$ be an orbifold vector bundle on $(M, \mathcal{V})$. For $U \in \mathcal{U}$, let $\wt{E^{pr}_U}$ be the maximal $K^E_U$-invariant sub-bundle of $\wt{E}_U$ on $U$. Then $(G_U, \wt{E^{pr}_U})$ defines a proper orbifold vector bundle on $(M, \mathcal{V})$, which is denoted by $E^{pr}$.

\begin{example}
 The (proper) orbifold tangent bundle $TM$ of an orbifold $M$ is defined by $(G_U , T \wt{U} \to \wt{U})$, for $U \in \mathcal{U}$.
\end{example}

\begin{definition}[$\mathscr{C}^k$ section]
 Let $E\to M$ be an orbifold bundle. A section $s \colon M \to E$ is called $\mathscr{C}^k$, for $k \in \N\cup\{\infty \}$, if for each $U \in \mathcal{U}$, $s_{|U}$ is covered by a $G^E_U$-invariant $\mathscr{C}^k$ section $\wt{s}_U \colon \wt{U} \to \wt{E}_U$. We denote by $\mathscr{C}^k(M,E)$ the space of $\mathscr{C}^k$ sections of $E$ on $M$.
 
 \textbf{When it entails no confusion, we will denote $\wt{s}_U$ simply by $s$.}
\end{definition}

\begin{remark}
\label{rem-smooth-objects}
 A smooth object on $M$ as a section, a Riemannian metric, a complex structure, etc... can be seen as an usual regular object on $M_{reg}$ such that its lift in any chart of the orbifold atlas can be extended to a smooth corresponding object.
\end{remark}

\begin{definition}[Integration]
If $M$ is oriented, we define the integral $\int_M \omega$ for a form $\omega$ over $M$ (i.e. a section of $\Lambda^\bullet(T^*M)$ over $M$) as follows: if $\mathrm{supp}(\omega) \subset U \in \mathcal{U}$, then
\begin{equation}
\label{defn-integration-eq}
\int_M \omega = \frac{1}{|G_U|}\int_{\wt{U}} \wt{\omega}_U.
\end{equation}
It is easy to see that the definition is independent of the chart. For general $\omega$ we extend the definition by using a partition of unity.

Note also that if $M$ is a Riemannian orbifold, there exists a canonical volume element $dv_M$ on $M$, which is a section of $\Lambda^{\dim M}(T^*M)$. Hence, we can also integrate functions on $M$.
\end{definition}

\begin{definition}[Metric structure on Riemannian orbifold]
 Let $(M,\mathcal{V})$ be a compact Riemannian orbifold. For $x,y\in M$, we define $d^M(x,y)$  by:
 \begin{multline}
d^M(x,y)= \inf_\gamma \Big\{ \textstyle{\sum_i}\int_{t_{i-1}}^{t_i} | \derpar{}{t}\tilde{\gamma}_i(t)|dt \:\, \Big | \: \gamma\colon[0,1]\to M,\, \gamma(0)=x,\, \gamma(1)=y, \text{ such that }\\
\text{there are } t_0=0<t_1<\dots<t_k=1 \text{ with } \gamma([t_{i-1},t_i])\subset U_i,\, U_i\in \mathcal{U},\\
 \text{and a } \smooth \text{ map } \tilde{\gamma}_i\colon [t_{i-1},t_i] \to \wt{U}_i \text{ which covers } \gamma|_{[t_{i-1},t_i]} \Big\}.
\end{multline}
Then $(M,d)$ is a metric space.
\end{definition}

%%%%%%%%%%%%%%%%%%%%%%%%%%%%%%%%%%%%%%%%%%%%%%%%%%%%%%%%%%%%%%%%%%%%%%%%%%%%%%%%%%%%%%%%%%%%%%%%%%%%%%%%%%%%%%%%%%%%%%%%%%%%%%%%%%%%%%%%%%%%%%%%%%%%%%%%%
\subsection{Kernels on orbifolds}
\label{Sect-kernel-orbifolds}

Let $(M,\mathcal{V})$ be a Riemannian orbifold and let $E$ be a proper orbifold vector bundle on $M$.

For any orbifold chart $(G_U,\wt{U})\overset{\tau_U}{\longrightarrow} U$, $U\in \mathcal{U}$, we will add a tilda $\tilde{\phantom{a}}$ to objects on $U$ to indicate the corresponding objects on $\wt{U}$.

Consider a section $\wt{\mathcal{K}}\in \smooth (\wt{U}\times \wt{U}, \mathrm{pr}_1^*\wt{E}\otimes \mathrm{pr}_2^* \wt{E}^*)$ such that
\begin{equation}
(g,1)\wt{\mathcal{K}}(g^{-1}\tilde{x}, \tilde{x}')=(1,g^{-1})\wt{\mathcal{K}}(\tilde{x},g\tilde{x}') \quad \text{ for any } g\in G_U,
\end{equation}
where the action of $G_U\times G_U$ on $\wt{E}_{\tilde{x}}\otimes \wt{E}^*_{\tilde{x}'}$ is given by $(g_1,g_2). u\otimes \xi=(g_1u)\otimes (g_2\xi)$. We can then define an operator $\wt{\mathcal{K}}\colon \smooth_0(\wt{U},\wt{E})\to \smooth(\wt{U},\wt{E})$ by
\begin{equation}
(\wt{\mathcal{K}}\tilde{s})(\tilde{x})=\int_{\wt{U}} \wt{\mathcal{K}}(\tilde{x},\tilde{x}')\tilde{s}(\tilde{x}')dv_{\wt{U}}(\tilde{x}') \quad \text{ for } \tilde{s}\in \smooth_0(\wt{U},\wt{E}).
\end{equation}

Recall that a section $s\in \smooth(U,E)$ is identified with a $G_U$-invariant section $\tilde{s}\in \smooth(\wt{U},\wt{E})$. Thus, we can define an operator $\mathcal{K}\colon \smooth_0(U,E)\to \smooth(U,E)$ by 
\begin{equation}
(\mathcal{K}s)(x)=\frac{1}{|G_U|}\int_{\wt{U}} \wt{\mathcal{K}}(\tilde{x},\tilde{x}')\tilde{s}(\tilde{x}')dv_{\wt{U}}(\tilde{x}') \quad \text{ for } s\in \smooth_0(U,E),
\end{equation}
where $\tilde{x}\in \tau_U^{-1}(x)$. Then the smooth kernel $\mathcal{K}(x,x')$ of the operator $\mathcal{K}$ with respect to $dv_M$ is given by
\begin{equation}
\mathcal{K}(x,x')=\sum_{g\in G_U} (g,1)\wt{\mathcal{K}}(g^{-1}\tilde{x}, \tilde{x}').
\end{equation}

%%%%%%%%%%%%%%%%%%%%%%%%%%%%%%%%%%%%%%%%%%%%%%%%%%%%%%%%%%%%%%%%%%%%%%%%%%%%%%%%%%%%%%%%%%%%%%%%%%%%%%%%%%%%%%%%%%%%%%%%%%%%%%%%%%%%%%%%%%%%%%%%%%%%%%%%%
\subsection{Complex orbifolds and Dolbeault cohomology}
\label{sect-complex-orbifolds}

Let $M$ be a compact complex orbifold of complex dimension $n$ and with complex structure $J$. Let $E$ be a holomorphic orbifold vector bundle on $M$.

Let $\0_M$ be the sheaf over $M$ of local $G_U$-invariant holomorphic functions over $\wt{U}$, for $U\in \mathcal{U}$. An element of $\0_M(M)$ is called an \emph{orbifold holomorphic function} on $M$.

Likewise, the local $G_U^E$-invariant sections of $\wt{E}$ over $\wt{U}$ define a sheaf $\0_M(E)$ over $M$. Let $H^\bullet(M,\0_M(E))$ be the cohomology of this sheaf. Notice that, by the definition, we have $\0_M(E)=\0_M(E^{pr})$. Thus without lost generality, we may and will assume that $E$ is a proper orbifold vector bundle on $M$.

Consider a section $s\in \smooth(M,E)$ and a local section $\tilde{s}\in \smooth(\wt{U},\wt{E})$ covering $s$ over $U$. Then $\db^{\wt{E}}\tilde{s}$ covers a section  of $T^{*(0,1)}X\otimes E$ over $U$, denoted by $\db^Es|_U$. The sections $\db^Es|_U$ for $U\in \mathcal{U}$ patch together to define a global section $\db^E s$ of $T^{*(0,1)}X\otimes E$ over $M$. In a similar way, we can ddefine $\db^E\alpha$ for $\alpha\in \Omega^{\bullet,\bullet}(M,E):=\smooth(M,\Lambda^{\bullet,\bullet}(T^*M)\otimes E)$. We thus obtain the Dolbeault complex
\begin{equation}
0\to \Omega^{0,0}(M,E) \overset{\db^E}{\to}\cdots \overset{\db^E}{\to} \Omega^{0,n}(M,E)\to 0.
\end{equation}
From the abstract de Rham theorem, there exists a canonical isomorphism (for more details, see \cite[Section 5.4.2]{ma-marinescu})
\begin{equation}
H^\bullet(\Omega^{0,\bullet}(M,E),\db^E)\simeq H^\bullet(M,\0_M(E)).
\end{equation}
In the sequel, we will denote both these cohomology groups simply by $H^\bullet(M,E)$.

Let $g^{TM}$ be a Riemannian metric on $TM$, with associated volume form $dv_M$, and let $h^E$ be a Hermitian metric on $E$. They induce a $L^2$ Hermitian product $\langle \cdot \, , \cdot \rangle$ on $\Omega^\bullet(M,E)$ given by 
\begin{equation}
\label{def-L^2-product}
\langle s_1,s_2\rangle = \int_M \langle s_1(x),s_2(x)\rangle _{\Lambda^{0,\bullet}(T^*X)\otimes E}\, dv_M(x).
\end{equation}
Let $\db^{E,*}$ be the formal adjoint of $\db^E$ for this $L^2$ product, and let 
\begin{equation}
\begin{aligned}
& D^E=\sqrt{2}(\db^E+\db^{E,*}), \\
& \square^E = \frac{1}{2} D^{E,2} = \db^E\db^{E,*}+\db^{E,*}\db^E.
\end{aligned}
\end{equation}
be, respectively, the associated Dolbeault-Dirac operator and Kodaira Laplacian. Then $\square^E$ a differential operator of order 2 acting on sections of $\Lambda^\bullet(T^*M)\otimes E$ (i.e., on each $U$, it is covered by a $G_U^{\Lambda^\bullet(T^*M)\otimes E}$-invariant differential operator of order 2 acting on $\smooth(\wt{U},\Lambda^\bullet(\wt{T^*M})\otimes \wt{E})$), which is formally self-adjoint  and elliptic. Moreover, it preserves the $\Z$-grading on $\Omega^\bullet(M,E)$.

A crucial point is that classical Hodge theory still holds in the present orbifold setting, see \cite[Proposition 2.2]{MaTAMS}.
\begin{theorem}[Hodge theory]
\label{Thm-Hodge}
 For any $q\in \N$, we have the following orthogonal decomposition
 \begin{equation}
\Omega^{0,q}(M,E) = \Ker(D^E|_{\Omega^{0,q}})\oplus \Im(\db^E|_{\Omega^{0,q-1}})\oplus \Im(\db^{E,*}|_{\Omega^{0,q+1}}).
\end{equation}
In particular, for $q\in \N$, we have the canonical isomorphism 
\begin{equation}
\Ker(D^E|_{\Omega^{0,q}})=\Ker(\square^{E}|_{\Omega^{0,q}}) \simeq H^q(M,E).
\end{equation}
\end{theorem}

Note that on orbifolds, we still have a unique Hermitian and holomorphic connection $\n^E$ associated with $E$ and $h^E$. We call it the Chern connection of $(E,h^E)$.

%%%%%%%%%%%%%%%%%%%%%%%%%%%%%%%%%%%%%%%%%%%%%%%%%%%%%%%%%%%%%%%%%%%%%%%%%%%%%%%%%%%%%%%%%%%%%%%%%%%%%%%%%%%%%%%%%%%%%%%%%%%%%%%%%%%%%%%%%%%%%%%%%%%%%%%%%
\section{Convergence of the heat kernel}
\label{cvce-noyau-Sect}

In the sequel, when we define objects on orbifolds, one can always think of them as being defined by standard objects which are $G_U$-invariant on each chart $(G_U,\wt{U})$.

Let $M$ be a compact complex orbifold of dimension $n$. We denote the complex structure of $M$ by $J$. Let $(L,h^L)$ an orbifold Hermitian holomorphic line bundle on $M$ and let $(E,h^E)$ be an orbifold Hermitian holomorphic vector bundle on $M$. Recall that by Section \ref{sect-complex-orbifolds} we may assume that $L$ and $E$ are proper. We denote the corresponding Chern connections by $\n^L$ and $\n^E$ respectively, and we denote their curvatures by $R^L$ and $R^E$.

We define the orbifold bundles $\E$ and  $\E_p$ over $M$ by
\begin{equation}
\begin{aligned}
& \E= \Wedge(T^*M)\otimes E,\\
& \E_{p}= \Wedge(T^*M)\otimes E \otimes L^p.
\end{aligned}
\end{equation}
Let $g^{TM}$ be a Riemannian metric on $TM$ which is compatible with $J$.  Then $\E$ and $\E_p$ are naturally equipped with the Hermitian metrics $h^{\E}$ and $h^{\E_p}$ induced by $g^{TM}$, $h^L$ and $h^E$. We endow $\smooth(M,\E_p)$ with the $L^2$ scalar product associated with $g^{TM}$, $h^L$ and $h^E$ as in \eqref{def-L^2-product}. As in section \ref{sect-complex-orbifolds}, we can define $D^{L^p\otimes E}$ and $\square^{L^p \otimes E}$. We will denote these operators respectively by $D_p$ and $\square_p$ for short.

Let  $e^{-u\square_p}$ be the heat kernel of $\square_p$ and let $e^{-u\square_p}(x,x')$ be its smooth kernel with respect to $dv_M(x')$. Concerning heat kernels on orbifolds, we refer the reader to \cite[Section 2.1]{MaTAMS}. 

In this section, we study the convergence as $p\to +\infty$ of the heat kernel. We follow the approach of \cite{ma-marinescu}.

\subsection{Localisation} 

Let $\e>0$ be a small number (smaller than the quarter of the injectivity radius of $(M,g^{TM})$). Let $f\colon \R \to [0,1]$ be a smooth even function such that
\begin{equation}
\label{def-f}
f(t)=\left \{
\begin{aligned}
&1 \text{ for } |t|<\e/2, \\
& 0 \text{ for } |t|>\e.
\end{aligned}
\right.
\end{equation}

For $u>0$ and $a\in \C$, set
\begin{equation}
\label{defFuGuHu}
\begin{aligned}
&\F_u(a)=\int_\R e^{iva}\exp(-v^2/2)f(v\sqrt{u})\frac{dv}{\sqrt{2\pi}}, \\
&\G_u(a)=\int_\R e^{iva}\exp(-v^2/2)(1-f(v\sqrt{u}))\frac{dv}{\sqrt{2\pi}}.
\end{aligned}
\end{equation}

These functions are even holomorphic functions. Moreover, the restrictions of $\F_u$ and $\G_u$  to $\R$ lie in the Schwartz space $\mathcal{S}(\R)$, and
\begin{equation}
\label{liensFGH}
 \F_u(vD_p)+\G_u(vD_p)=\exp\left( -\frac{v^2}{2}D_p^2\right) \text{ for }v>0.
\end{equation}

Let $\G_u(vD_p)(x,x')$ be the smooth kernel of $\G_u(vD_p)$ with respect to $dv_M(x')$. Then for any $m\in \N$, $u_0>0$, $\e>0$, there exist  $C>0$ and $N\in \N$ such that for any $u> u_0$ and any $p\in \N^*$,
\begin{equation}
\label{the-pb-is-local-eq}
\left|  \G_{\frac{u}{p}} \left(\sqrt{u/p}D_p\right)(\cdot\, ,\cdot) \right|_{\mathscr{C}^m(M\times M)} \leq Cp^N  \exp \left( -\frac{\e^2p}{16u} \right).
\end{equation}

This is proved in the same way as \cite[Proposition 1.6.4]{ma-marinescu}, the only difference is that when we use Sobolev norms or inequality on an open $U$, we in fact have to use them on $\wt{U}$ for the pulled-back operators.

As pointed out in \cite[Section 6.6]{MaTAMS}, the property of finite propagation speed of solutions of hyperbolic equations still holds on an orbifold (see the proof in \cite[Appendix D.2]{ma-marinescu}). Thus, $\F_{\frac{u}{p}} \left(\sqrt{u/p}D_p\right)(x,x')$ vanishes if $d^M(x,x')\geq \e$ and $\F_{\frac{u}{p}} \left(\sqrt{u/p}D_p\right)(x,\cdot)$ only depends on the restriction of $D_p$ to the ball $B^M(x,\e)$. This, together with \eqref{liensFGH} and \eqref{the-pb-is-local-eq}, implies that the problem of the asymptotic of $e^{-u\square_p}(x,\cdot)$ is local. 

More precisely, this means that, for $x_0\in M$ fixed, one can trivialize the various bundles over $U_{x_0}$ and replace $M$ by $M_0 = \C^n/G_{x_0} \supset  \wt{U}_{x_0}/G_{x_0}= U_{x_0}$ (using a local chart as in Lemma \ref{Lem-carte-lineaire}). Then we construct a metric $g^{TM_0}$ on $M_0$ and an operator $L_{p,x_0}$ acting on $\E_{p,x_0}$ over $M_0$ such that $g^{TM_0}$ (resp. $L_{p,x_0}$) coincides with $g^{TM}$ (resp. $\square_p$) near $x_0=0$ and such that its lift $\wt{g^{TM_0}}$ (resp. $\wt{L}_{p,x_0}$) on $\wt{M_0}=\C^n$ is the constant metric $(\wt{g^{TM}}_{U_{x_0}})_{\tilde{x}=0}$ (resp. the usual Laplacian on $\C^n$ for this metric) away from 0. Then we can approximate the heat kernel of $\square_p$ by the one of $L_{p,x_0}$, see equation \eqref{Eq-localisation-calcul-noyau} below. \\

We now give the details of these constructions, following \cite[Section 1.6.3]{ma-marinescu}. 

By \cite[(1.2.61) and (1.4.27)]{ma-marinescu}, the Levi-Civita connection $\n^{TM}$ on $(M,g^{TM})$, the complexe structure $J$ of $M$ and the Chern connection $\n^E$ induce a Hermitian connexion $\n^{B,\E}$ on $(\E,h^{\E})$ which preserve the $\Z$-grading. This connection is called the Bismut connection. Let $\Delta^{B,\E}$ be the associated Bochner Laplacian, that is
\begin{equation}
\Delta^{B,\E} = -\sum_{i=1}^{2n}\Big( (\n^{B,\E}_{e_i})^2-\n^{B,\E}_{\n^{TM}_{e_i}e_i}\Big),
\end{equation}
with $\{e_i\}$ an orthonormal frame of $TM$. 

Let $\{w_j\}_{j=1}^n$ be a local  orthonormal frame of $T^{(1,0)}M$ (with the metric induced by $g^{TM}$) with dual frame $\{w^j\}_{j=1}^n$. Set
 \begin{equation}
 \label{Eq-def-omega_d-tau}
\begin{aligned}
&\omega_d = -\sum_{k,\ell} R^{L} (w_k,\bw_\ell) \bw^\ell\wedge i_{\bw_k}, \\
&\tau = \sum_j R^{L} (w_j,\bw_j).
\end{aligned}
\end{equation}

From Bismut's Lichnerowicz formula (see \cite[Theorem 1.4.7]{ma-marinescu}), which still holds here because the computations involved are local and hence can be carried on orbifold charts, there exists a self-adjoint section $\Phi_E$  of $\End(\Wedge(T^*M)\otimes E))$ such that
\begin{equation}
\label{lichnerowicz}
\square_p = \frac{1}{2} \Delta^{B,\E} +\omega_d +\frac{1}{2}\tau+\Phi_E.
\end{equation}

We fix $x_0\in M$ and $\e>0$ smaller than the quarter of the injectivity radius of $M$. In the sequel we will use a local chart $(G_{x_0},\wt{U}_{x_0})$ near $x_0$ as in Lemma \ref{Lem-carte-lineaire}. Note that we then have $T_{x_0}M\simeq \C^n/G_{x_0}$.

We denote by $B^M(x_0 ,4\e)$ and $B^{T_{x_0}M}(0,4\e)$ the open balls in $M$ and $T_{x_0}M$ with center $x_0$ and 0 and radius $4\e$, respectively. The exponential map $T_{x_0}M \owns Z \mapsto \exp_{x_0}^M (Z) \in M$ is a diffeomorphism from $B^{T_{x_0}M}(0,4\e)$ on $B^M(x_0 ,4\e)$. From now on, we identify $B^{T_{x_0}M}(0,\e)$ and $B^M(x_0 ,4\e)$.

For $Z\in B^{T_{x_0}M}(0,4\e)$, we identify $(L_Z,h^L_Z)$ and $(\E_Z,h^{\E}_Z)$ to $(L_{x_0},h^L_{x_0})$ and $(\E_{x_0},h^{\E}_{x_0})$ by the parallel transport with respect to $\n^L$ and $\n^{B,\E}$ along the ray $u\in [0,1]\mapsto uZ$. We denote the corresponding connection forms by $\Gamma^L$ and $\Gamma^{\E}$. Note that $\Gamma^L$ and $\Gamma^{\E}$ are skew-adjoint with respect to $h^L_{x_0}$ and $h^{\E}_{x_0}$.

Let $\rho \colon \R\to [0,1]$ be a smooth even function such that
\begin{equation}
\rho(v) =1 \text{ if } |v|<2 \quad \text{and} \quad \rho(v) =0 \text{ if } |v|>4.
\end{equation}

We denote by $\n_{\wt{V}}$ the ordinary differentiation operator on $\wt{M_0}=\C^n$ in the direction $\wt{V}$.  Then $\n$ is $G_{x_0}$-equivariant since $G_{x_0}$ acts linearly on $\C^n$, thus it induces an operator still denoted by $\n$ on $M_0= \C^n/G_{x_0}$. Set
\begin{equation}
\n^{\E_{p,x_0}} = \n +\rho(|Z|/\e) (p\Gamma^L+\Gamma^\E)(Z),
\end{equation}
which is a Hermitian connection on the trivial bundle $(\E_{p,x_0},h^{\E_p}_{x_0})$ over $M_0\simeq T_{x_0}M$, where the identification is given by
\begin{equation}
[Z_1,\dots, Z_{2n}]\in \R^{2n}/G_{x_0} \mapsto \Big[ \textstyle{\sum_{i=1}^{2n}}Z_i\wt{e_i} \Big] \in T_{x_0}M.
\end{equation}
Here, $\{\wt{e_i}\}_i$ is an orthonormal basis of $(\wt{TM}_{U_{x_0}})_{\tilde{x}=0}$.

Let $g^{TM_0}$ be the metric on $M_0$ which coincides with $g^{TM}$ on $B^{T_{x_0}M}(0,2\e)$ and such that its lift $\wt{g^{TM_0}}$ on $\wt{M_0}$ is the constant metric $(\wt{g^{TM}}_{U_{x_0}})_{\tilde{x}=0}$ outside of $B^{\wt{M_0}}(0,4\e)$. Let $dv_{\wt{M_0}}$ be the Riemannian volume of $\wt{g^{TM_0}}$.

Let $\Delta^{\E_{p,x_0}}$ be the Bochner Laplacian associated with $\n^{\E_{p,x_0}}$ and $g^{TM_0}$. Set
\begin{equation}
L_{p,x_0} = \frac{1}{2}\Delta^{\E_{p,x_0}} -p\rho(|Z|/\e)(\omega_{d,Z}+\frac{1}{2}\tau_Z)-\rho(|Z|/\e)\Phi_{E,Z}.
\end{equation}
Then $L_{p,x_0}$ is a self-adjoint operator for the $L^2$-product induced by $h^{\E_p}_{x_0}$ and $g^{TM_0}$, and coincides with $\square_p$ on $B^{T_{x_0}M}(0,2\e)$. We denote by $\wt{L}_{p,x_0}$ the $G_{x_0}$-invariant lift of $L_{p,x_0}$ on $\wt{M_0}$.

From the above discussion, and using the same technics as in \cite[Lemma 1.6.5]{ma-marinescu} (because, as said before, the property of finite propagation speed of solutions of hyperbolic equations still holds on orbifolds) we can then prove that for $(x,x') \in B^M(x_0,\e/2)$ corresponding to $(Z,Z')\in M_0$, there exist $C>0$ and $K\in \N$ such that
\begin{equation}
\label{Eq-localisation-calcul-noyau}
\left| e^{-\frac{u}{p}\square_p}(x,x')-e^{-\frac{u}{p}L_{p,x_0}}(Z,Z') \right|\leq Cp^Ke^{-\frac{\e^2p}{16u}},
\end{equation}
 where $e^{-\frac{u}{p}L_{p,x_0}}(Z,Z')$ is the smooth kernel of $e^{-\frac{u}{p}L_{p,x_0}}$ with respect to $dv_{M_0}(Z')$, the Riemannian volume of $g^{TM_0}$.

Now, as in Section \ref{Sect-kernel-orbifolds}, we have
\begin{equation}
\label{Eq-noyau-et-releve}
e^{-\frac{u}{p}L_{p,x_0}}(Z,Z')= \sum_{g\in G_{x_0}} (g,1)e^{-\frac{u}{p}\wt{L}_{p,x_0}}(g^{-1}\wt{Z}, \wt{Z}'),
\end{equation}
where $e^{-\frac{u}{p}\wt{L}_{p,x_0}}(\tilde{Z},\tilde{Z}')$ is the smooth kernel of $e^{-\frac{u}{p}\wt{L}_{p,x_0}}$ with respect to $dv_{\wt{M_0}}(\tilde{Z}')$.

Thus, from \eqref{Eq-localisation-calcul-noyau} and \eqref{Eq-noyau-et-releve}, we have to study the asymptotic of $e^{-\frac{u}{p}\wt{L}_{p,x_0}}(\tilde{Z}, \tilde{Z}')$, which is a kernel on a honest vector space. Note that, even if we want to restrict $e^{-\frac{u}{p}\square_p}$ to the diagonal, we have to study the off-diagonal asymptotic of $e^{-\frac{u}{p}\wt{L}_{p,x_0}}$.

 \subsection{Rescaling} 
 As in \cite{ma-marinescu}, we will rescale the variables in $\wt{M_0}$.
 
 \textbf{In this paragraph, we work on $\wt{U}_{x_0}$ and for any bundle $F$ on $M$ we will denote $\wt{F}_{U_{x_0}}$ simply by $\wt{F}$. Likewise, we will drop the subscript $U_{x_0}$ in the notations of lifts to $\wt{F}$ of objects on $F$.}
 
 Let $S_{\wt{L}}$ be a $G_{x_0}$-invariant unit vector of $\wt{L}|_0$.  Using $S_{\wt{L}}$ and the above discussion, we get an isometry $\wt{\E}_{p,x_0}\simeq \wt{\E}_{x_0}$. Thus, $\wt{L}_{p,x_0}$ can be seen as an operator on $\wt{\E}_{x_0}$. Note that our formulas will not depend on the choice of $S_L$ as the isomorphism $\End(\wt{\E}_{p,x_0})\simeq \End(\wt{\E}_{x_0})$ is canonical.

Let $dv_{\wt{TM}}$ be the Riemannian volume of $(\wt{M_0},\wt{g^{TM}}|_{0})$. Let $\wt{\kappa}$ be the smooth positive function defined by
\begin{equation}
\label{def-kappa}
dv_{\wt{M_0}}(\wt{Z}) = \wt{\kappa}(\wt{Z}) dv_{\wt{TM}}(\wt{Z}),
\end{equation}
with $\wt{\kappa}(0)=1$. Note that this definition is compatible with \eqref{def-kappa-intro} near 0, which will be \emph{in fine}  the only region of interest.

Let $R^{\wt{L}}$ be the Chern curvature of $(\wt{L},\wt{h^{L}})$. Let $\wt{\omega}_d$ and $\wt{\tau}$ be defined from $R^{\wt{L}}$ as $\omega_d$ and $\tau$  were defined from $R^L$ in \eqref{Eq-def-omega_d-tau}. Then $\wt{\omega}_d$ and $\wt{\tau}$ are in fact the lifts of $\omega_d$ and $\tau$. 

Recall that $\n_V$ is the ordinary differentiation operator on $\wt{M_0}=\C^n$ in the direction $V$.  

We will now make the change of parameter $t=\frac{1}{\sqrt{p}}\in\, ]0,1]$. 

\begin{definition}
For $s\in \smooth(\wt{M_0}, \E_{x_0})$ and $\tilde{Z}\in\C^{n}$ set
\begin{equation}
\label{Eq-def-rescaled}
\begin{aligned}
&(S_ts)(\tilde{Z})  = s(\tilde{Z}/t), \\
& \n_0= \nabla + \frac{1}{2} R^{\wt{L}}_{0}( \tilde{Z}, \cdot) , \\
&\wt{\LL}_{t}=   t^2 S_t^{-1}  \wt{\kappa}^{1/2} \wt{L}_{p,x_0} \wt{\kappa}^{-1/2} S_t, \\
&\wt{\LL}_{0} = - \frac{1}{2}\sum_{i=1}^{2n} \left(\n_{0,\wt{e_i}}\right)^2-\wt{\omega}_{d,0}-\frac{1}{2}\wt{\tau}_{0}.
\end{aligned}
\end{equation}
\end{definition}

Then, exactly as in \cite[Lemma 1.6.6]{ma-marinescu}, our constructions imply that $\wt{\LL}_t=\wt{\LL}_0+O(t)$. 

Let $e^{-u\wt{\LL}_{t}}(\wt{Z},\wt{Z}')$ be the smooth kernel of $e^{-u\wt{\LL}_{t}}$ with respect to $dv_{\wt{TM}}(\wt{Z}')$ (for $t>0$ or $t=0$). Now, as we are working on a vector space, we can apply all the results of \cite{MR2215454} (see also \cite[Section 4.2.2]{ma-marinescu}) to $\wt{\LL}_t$ and $\wt{\LL}_0$, and we get the following full off-diagonal convergence (see \cite{MR2215454} or \cite[Theorem 4.2.8]{ma-marinescu}).

\begin{theorem}
\label{Thm-cve-noyau-Lt}
There exist $C,C'>0$ and $N\in\N$ such that for any $m,m'\in \N$ and $u_0>0$, there are $C_{m,m'}>0$ and $N\in \N$ such that for any $t\in\, ]0,t_0]$, $u\geq u_0$ and $\wt{Z},\wt{Z}'\in \wt{M_0}$ with $|\wt{Z}|,|\wt{Z}'|\leq 1$
\begin{multline}
\label{Eq-cve-noyau-Lt}
\sup_{|\alpha|,|\alpha'|\leq m}  \left| \frac{\partial^{|\alpha|+|\alpha'|}}{\partial {\wt{Z}}^{\alpha}\partial \wt{Z}'^{\alpha'}} \left(e^{-u\wt{\LL}_{t}}-e^{-u\wt{\LL}_{0}
}\right)(\wt{Z},\wt{Z}')\right|_{\mathscr{C}^{m'}(M)} \\
\leq C_{m,m'}t\big(1+|\wt{Z}|+|\wt{Z}'|  \big)^{N} \exp\Big(C u -\frac{C'}{u}|\wt{Z}-\wt{Z}'|^2\Big),
\end{multline}
where $|\cdot|_{\mathscr{C}^{m'}(M)}$ denotes the $\mathscr{C}^{m'}$-norm with respect to the parameter $x_0\in M$ used to define the operators $\wt{\LL}_t$ and $\wt{\LL}_0$ on $\C^n$.
\end{theorem}

\subsection{Conclusion}

From \eqref{Eq-def-rescaled}, a change of variable gives that 
\begin{equation}
e^{-\frac{u}{p}\wt{L}_{p,x_0}}(\wt{Z},\wt{Z}')=p^n e^{-u\wt{\LL}_t}(\wt{Z}/t,\wt{Z}'/t)\wt{\kappa}^{-1/2}(\wt{Z})\wt{\kappa}^{-1/2}(\wt{Z}').
\end{equation}
Thus, with Theorem \ref{Thm-cve-noyau-Lt}, we infer that for any multi-index $\alpha$ with $|\alpha|\leq m$ and for $|\wt{Z}|$ small, 
\begin{multline}
\label{cvce-noyau-Lp-Ep}
\left| \derpar{^{|\alpha|}}{\wt{Z}^\alpha} \left(p^{-n}e^{-\frac{u}{p}\wt{L}_{p,x_0}}(g^{-1}\wt{Z},\wt{Z}) - e^{-u\wt{\LL}_0}(\sqrt{p}g^{-1}\wt{Z},\sqrt{p}\wt{Z})\wt{\kappa}^{-1}(\wt{Z}) \right)\right|_{\mathscr{C}^{m'}(M)} \\
 \leq Cp^{\frac{m-1}{2}}\big(1+\sqrt{p}|\wt{Z}|  \big)^N e^{ -cp|\wt{Z}-g^{-1}\wt{Z}|^2}.
\end{multline}

We define $\dot{R}^{\wt{L}}\in \End(T^{(1,0)}\wt{M_0})$ by $g^{TM}(\dot{R}^{\wt{L}}u,\ol{v}) =R^{\wt{L}}(u,\ol{v})$ for $u,v\in T^{(1,0)}\wt{M_0}$. We extend $\dot{R}^{\wt{L}}$ to $T\wt{M_0}\otimes \C=T^{(1,0)}\wt{M_0}\oplus T^{(0,1)}\wt{M_0}$ by setting $\dot{R}^{\wt{L}}\ol{v}=-\ol{\dot{R}^{\wt{L}}v}$. Then $\ic\dot{R}^{\wt{L}}_{x}$ induces an anti-symmetric endomorphism of $T\wt{M_0}$. Then from the formula for the heat kernel of a harmonic oscillator (see \cite[(E.2.4), (E.2.5)]{ma-marinescu} for instance), we find:
\begin{multline}
\label{chaleur-L}
e^{-u\wt{\LL}_0}(g^{-1}\wt{Z},\wt{Z})=\frac{1}{(2\pi)^{n}}\frac{\det(\dot{R}_{0}^{\wt{L}})e^{u\wt{\omega}_{d,0}}}{\det\big(1-\exp(-u\dot{R}_{0}^{\wt{L}})\big)}\otimes \Id_{\wt{E}_0}\\
\times \exp\Bigg\{-\bigg\langle\frac{\dot{R}_{0}^{\wt{L}}/2}{\mathrm{th}(u\dot{R}_{0}^{\wt{L}}/2)} \wt{Z},\wt{Z} \bigg\rangle+\bigg\langle\frac{\dot{R}_{0}^{\wt{L}}/2}{\mathrm{sh}(u\dot{R}_{0}^{\wt{L}}/2)} e^{u\dot{R}_{0}^{\wt{L}}/2}g^{-1}\wt{Z},\wt{Z} \bigg\rangle\Bigg\}.
\end{multline}
Here, we use the convention that if an eigenvalue of $\dot{R}^{\wt{L}}_{0}$ is zero, then its contribution to the above term is $\wt{v}\mapsto \frac{1}{2\pi u}e^{-\frac{1}{2u}|g^{-1}\wt{v}-\wt{v}|^2}$.

We are now able to prove Theorems \ref{asymp-away-sing-Thm} and \ref{asymp-near-sing-Thm}.

\begin{proof}[Proof of Theorem  \ref{asymp-away-sing-Thm}]
If $x_0$ is in $M_{reg}$, we have $G_{x_0}=\{1\}$ and the objects with or without tildas coicindes. Thus, from \eqref{Eq-localisation-calcul-noyau}, \eqref{Eq-noyau-et-releve}, and \eqref{cvce-noyau-Lp-Ep}, \eqref{chaleur-L} applied at $\wt{Z}=0$, we get Theorem \ref{asymp-away-sing-Thm}.
  \end{proof}

\begin{proof}[Proof of Theorem  \ref{asymp-near-sing-Thm}]
We will use here the notations given in the introduction of this paper (before the statement of Theorem  \ref{asymp-near-sing-Thm}). In Particular, for $g\in G_{x_0}$, we have a decomposition $\wt{Z}=(\wt{Z}_{1,g},\wt{Z}_{2,g})$ where $\wt{Z}_{1,g}$ is in the fixed-point set of $g$ and $\wt{Z}_{2,g}$ is in the normal bundle of this set.

Let us fix $g\in G_{x_0}$. The idea is to apply the results of this Section \ref{cvce-noyau-Sect} but replacing the base-point $0\in \wt{U}_{x_0}$ by $\wt{Z}_{1,g}$. In order to stress on the dependence on $\wt{Z}_{1,g}$, we will add subscript to the various objects introduced above but defined with the base-point $\wt{Z}_{1,g}$, e.g., $\kappa_{\wt{Z}_{1,g}}$, $\wt{\LL}_{0,\wt{Z}_{1,g}}$, etc... We can then make \eqref{cvce-noyau-Lp-Ep} more precise: observe that there is $c_0>0$ such that for each $g\in G_{x_0}$, $|\wt{Z}_{2,g}-g^{-1}\wt{Z}_{2,g}|\geq c_0 |\wt{Z}_{2,g}|$, and thus for $m,\ell \in \N$ and $|\alpha'|\leq \ell$ we have some constants $c,C>0$ such that
\begin{multline}
\label{cvce-noyau-Lp-avec_g-Ep}
\sup_{|\alpha|\leq m}\left| \derpar{^{|\alpha|}}{\wt{Z}^\alpha_{1,g}}\derpar{^{|\alpha'|}}{\wt{Z}^{\alpha'}_{2,g}} \left(p^{-n}e^{-\frac{u}{p}\wt{L}_{p,x_0}}(g^{-1}\wt{Z},\wt{Z}) - e^{-u\wt{\LL}_{0,\wt{Z}_{1,g}}}(\sqrt{p}g^{-1}\wt{Z}_{2,g},\sqrt{p}\wt{Z}_{2,g})\wt{\kappa}_{\wt{Z}_{1,g}}^{-1}(\wt{Z}_{2,g}) \right)\right| \\
 \leq Cp^{\frac{\ell-1}{2}}\big(1+\sqrt{p}|\wt{Z}_{2,g}|  \big)^N e^{ -cp|\wt{Z}_{2,g}|^2}.
\end{multline}
In particular, we find that if $g=1$ then $\wt{Z}=\wt{Z}_{1,g}$ and $\wt{Z}_{2,g}=0$ and thus
\begin{equation}
\label{cvce-noyau-Lp-g=1-Ep}
\sup_{|\alpha|\leq m}\left| \derpar{^{|\alpha|}}{\wt{Z}^\alpha} \left(p^{-n}e^{-\frac{u}{p}\wt{L}_{p,x_0}}(\wt{Z},\wt{Z}) -  e^{-u\wt{\LL}_{0,\wt{Z}}}(0,0)\right)\right|  \leq Cp^{-1/2}.
\end{equation}

 Note that the image in the quotient of the union $\cup_{g\neq1}\wt{U}_{x_0}^g$ is precisely $ M_{sing}\cap U_{x_0}$. In particular, if $Z$ is the image of $\wt{Z}$, then we have  $|\wt{Z}_{2,g}|\geq d(Z,M_{sing})$ for $g\in G_{x_0}\setminus\{1\}$. From this remark and equations \eqref{Eq-localisation-calcul-noyau}, \eqref{Eq-noyau-et-releve}, \eqref{cvce-noyau-Lp-avec_g-Ep} and \eqref{cvce-noyau-Lp-g=1-Ep}  we find that 
\begin{multline}
\label{cvce-noyau-square_p-Ep}
\sup_{|\alpha|\leq \ell} \Bigg|\derpar{^{|\alpha|}}{\wt{Z}^\alpha}\Bigg(p^{-n}e^{-\frac{u}{p}\square_p}(\wt{Z},\wt{Z})-e^{-u\wt{\LL}_{0,\wt{Z}}}(0,0)\\
 -\sum_{\substack{g\in G_{x_0} \\ g\neq 1}}(g,1). e^{-u\wt{\LL}_{0,\wt{Z}_{1,g}}}(\sqrt{p}g^{-1}\wt{Z}_{2,g},\sqrt{p}\wt{Z}_{2,g})\wt{\kappa}_{\wt{Z}_{1,g}}^{-1}(\wt{Z}_{2,g})\Bigg)\Bigg|\\
 \leq Cp^{-1/2}+ Cp^{\frac{\ell-1}{2}}(1+\sqrt{p}\,d(Z,M_{sing}))^Ne^{-cp\,d(Z,M_{sing})^2}.
\end{multline}

To conclude, we get Theorem  \ref{asymp-near-sing-Thm} thanks to the definiton of $e^{i\theta_g}$ and $g^E$, and \eqref{def-Lim-Eq}, \eqref{def-E-Eq}, \eqref{chaleur-L} and \eqref{cvce-noyau-square_p-Ep}, noticing that $\dot{R}^{\wt{L}}$  and $\wt{\omega}_{d}$ coincide with the invariant lifts of $\dot{R}^{L}$ and $\omega_{d}$.
  \end{proof}

%%%%%%%%%%%%%%%%%%%%%%%%%%%%%%%%%%%%%%%%%%%%%%%%%%%%%%%%%%%%%%%%%%%%%%%%%%%%%%%%%%%%%%%%%%%%%%%%%%%%%%%%%%%%%%%%%%%%%%%%%%%%%%%%%%%%%%%%%%%%%%%%%%%%%%%%%%%%%%%%%%%%%%%%%%%%%%%%%%%%%%%%%%%%%%%%%%%%%%%%%%%%%%%%%%%%%%%%%%%%%%%%%%%%%%%%%%%%%%%%%%%%%%%%%%%%%%%%%%%%%%%%%%%%%%%%%%%%%%%%%%%%%%%
\section{Proof of the inequalities}
\label{Sect-proof}

In this section, we will first prove Theorem  \ref{Euler-et-chaleur}, and then show how to use it in conjunction with the convergence of the heat kernel proved in Section \ref{cvce-noyau-Sect} to get Theorem \ref{OHMI}. The method is inspired by \cite{MR886814} (see also \cite[Section 1.7]{ma-marinescu}).

\begin{proof}[Proof of Theorem  \ref{Euler-et-chaleur}]
If $\lambda$ is an eigenvalue of $\square_p$ acting on $\Omega^{0,j}(M,L^p\otimes E)$, we denote by $F_j^\lambda$ the corresponding finite-dimensional eigenspace. As $\db^{L^p\otimes E}$ and $\db^{L^p\otimes E,*}$  commute with $\square_p$, we deduce that 
\begin{equation}
\db^{L^p\otimes E}(F^\lambda_j)\subset F^\lambda_{j+1} \qquad \text{and} \qquad  \db^{L^p\otimes E,*}(F^\lambda_j)\subset F^\lambda_{j-1}.
\end{equation}
 
As a consequence, we have a complexe
\begin{equation}
\label{complexe-espace-propre}
0\longrightarrow F^\lambda_0 \overset{\db^{L^p\otimes E}}{\longrightarrow}F^\lambda_1 \overset{\db^{L^p\otimes E}}{\longrightarrow} \cdots \overset{\db^{L^p\otimes E}}{\longrightarrow}F^\lambda_n \longrightarrow 0.
\end{equation}
If $\lambda=0$, we have  $F^0_j \simeq  H^j(M,L^p\otimes E)$ by Theorem \ref{Thm-Hodge}. If $\lambda>0$, then the complex \eqref{complexe-espace-propre} is exact. Indeed, if $\db^{L^p\otimes E}s=0$  and $s\in F^\lambda_j$, then
\begin{equation}
s=\lambda^{-1}\square_ps = \lambda^{-1}\db^{L^p\otimes E}\db^{L^p\otimes E,*}s \in \Im(\db^{L^p\otimes E}).
\end{equation}
In particular, we get for $\lambda>0$ and $0\leq q \leq n$
\begin{equation}
\label{sum-dim(Flambdaj)}
\sum_{j=0}^q(-1)^{q-j}\dim \, F^\lambda_j = \dim \big( \db^{L^p\otimes E}(F^\lambda_q) \big)\geq 0,
\end{equation}
with equality if $q=n$. 

Now, by Theorem \ref{Thm-Hodge}
\begin{equation}
\label{trace-et-vap}
\tr|_{\Omega^j}[e^{-\frac{u}{p}\square_p}]= \dim(H^j(M,L^p\otimes E) +\sum_{\lambda>0}e^{-\frac{u}{p}\lambda}\dim\, F^\lambda_j.
\end{equation}
Finally, \eqref{sum-dim(Flambdaj)} and \eqref{trace-et-vap} entail \eqref{Euler-et-chaleur-eq}.
  \end{proof}

We can now conclude.
\begin{proof}[Proof of Theorem  \ref{OHMI}]
Let $\{x_i\}_{1\leq i \leq m}$ be a finite set of points of $M_{sing}$ such that the corresponding local charts $(G_{x_i},\wt{U}_{x_i})$ (as in Lemma \ref{Lem-carte-lineaire}) with $\wt{U}_{x_i}\subset \C^n$ satisfy 
\begin{equation}
B^{\wt{U}_{x_i}}(0,2\e)\subset \wt{U}_{x_i} \quad \text{and} \quad M_{sing}\subset \bigcup_{i=1}^m  W_i\,, \quad  W_i:=B^{\wt{U}_{x_i}}\big(0,\frac{\e}{4}\big)/G_{x_i}.
\end{equation}
Here $\e$ is as in Section \ref{cvce-noyau-Sect}. Let $W_0$ be an open neighborhood of the complementary of $\bigcup_{i=1}^m W_i$ wich is relatively compact  in $M_{reg}$. Let $\{\psi_k\}_{0\leq k \leq m}$ be a partition of the unity subordinated to $\{W_k \}_{ 0\leq k\leq m}$.

In the sequel, we denote by $\tr_{\Lambda^{0,q}}$ the trace on either $\Lambda^{0,q}(T^*M)\otimes L^p \otimes E$ or $\Lambda^{0,q}(T^*M)$. For $0\leq q \leq n$, we have
\begin{equation}
\label{trace-comme-int-de-trace}
\begin{aligned}
\tr_q\big[e^{-\frac{u}{p}\square_p}\big]&=\int_{M} \tr_{\Lambda^{0,q}}[e^{-\frac{u}{p}\square_p}(x,x)]dv_M(x)\\
&=\sum_{k=0}^m \int_{M} \psi_k(x)\tr_{\Lambda^{0,q}}[e^{-\frac{u}{p}\square_p}(x,x)]dv_M(x).
\end{aligned}
\end{equation}

From Theorem  \ref{asymp-away-sing-Thm}, we know that for $p\to \infty$
\begin{multline}
\label{trace-avec-k=0-Eq}
p^{-n}\int_{M} \psi_0(x)\tr_{\Lambda^{0,q}}[e^{-\frac{u}{p}\square_p}(x,x)]dv_M(x) = \\
 \frac{\mathrm{rk}(E)}{(2\pi)^{n}}\int_M \psi_0(x) \frac{\det(\dot{R}_{x}^{L})\tr_{\Lambda^{0,q}}[e^{u\omega_{d,x}}]}{\det\big(1-\exp(-u\dot{R}_{x}^{L})\big)}dv_M(x)+o(1).
\end{multline}

For $1\leq k \leq m$, we know from Theorem  \ref{asymp-near-sing-Thm} that for $p\to \infty$
\begin{multline}
\label{trace-avec-k>0-Eq}
p^{-n}\int_{M} \psi_k(x)\tr_{\Lambda^{0,q}}[e^{-\frac{u}{p}\square_p}(x,x)]dv_M(x) = \\
 \sum_{\substack{g\in G_{x_k} \\ g\neq 1}}\frac{1}{|G_{x_k}|}\int_{\{|\wt{Z}|\leq \frac{\e}{4}\}} \psi_k(\wt{Z})e^{ip\theta_g}g^{E}(\wt{Z}_{1,g})\mathcal{L}im_u(\wt{Z}_{1,g})\mathcal{E}_{g,\wt{Z}_{1,g}}(u,\sqrt{p}\wt{Z}_{2,g})dv_{\wt{TM}}(\wt{Z}) \\
 +\frac{\mathrm{rk}(E)}{(2\pi)^{n}}\int_M \psi_k(x) \frac{\det(\dot{R}_{x}^{L})\tr_{\Lambda^{0,q}}[e^{u\omega_{d,x}}]}{\det\big(1-\exp(-u\dot{R}_{x}^{L})\big)}dv_M(x)+o(1).
\end{multline}

However, for $g\neq 1$, observe that
\begin{multline}
\label{Eq-int-negligeable}
\int_{\{|\wt{Z}|\leq \frac{\e}{4}\}} \psi_k(\wt{Z})e^{ip\theta_g}g^{E}(\wt{Z}_{1,g})\mathcal{L}im_u(\wt{Z}_{1,g})\mathcal{E}_{g,\wt{Z}_{1,g}}(u,\sqrt{p}\wt{Z}_{2,g})dv_{\wt{TM}}(\wt{Z}) \\
= p^{-\frac{\dim N_{x_k,g}}{2}}\int_{A(p,\e)} \psi_k\Big(\big(\wt{Z}_{1,g},\frac{\wt{Z}'_{2,g}}{\sqrt{p}}\big)\Big)e^{ip\theta_g}g^{E}(\wt{Z}_{1,g})\mathcal{L}im_u(\wt{Z}_{1,g})\mathcal{E}_{g,\wt{Z}_{1,g}}(u,\wt{Z}'_{2,g})dv_{\wt{TM}}(\wt{Z}), 
\end{multline}
where $A(p,\e)=\{|\wt{Z}_{1,g}|^2+\frac{1}{p} |\wt{Z}'_{2,g}|^2\leq \frac{\e}{4}\}$. Now, as we have $\langle (1-g^{-1})\wt{Z}'_{2,g},\wt{Z}'_{2,g}\rangle \geq c_1 |\wt{Z}'_{2,g}|$ with $c_1>0$, we can see that $\mathcal{E}_{g,\wt{Z}_{1,g}}(u,\wt{Z}'_{2,g})$ is exponentially decaying as $|\wt{Z}'_{2,g}|\to\infty$. Thus, there is $C>0$ such that
\begin{equation}
\left| \int_{A(p,\e)} \psi_k\Big(\big(\wt{Z}_{1,g},\frac{\wt{Z}'_{2,g}}{\sqrt{p}}\big)\Big)e^{ip\theta_g}g^{E}(\wt{Z}_{1,g})\mathcal{L}im_u(\wt{Z}_{1,g})\mathcal{E}_{g,\wt{Z}_{1,g}}(u,\wt{Z}'_{2,g})dv_{\wt{TM}}(\wt{Z}) \right| \leq C.
\end{equation}
as a consequence,  since  $\dim N_{x_k,g}>0$ for $g\neq 1$, we deduce that all the integrals \eqref{Eq-int-negligeable}  for $1\leq k\leq m$ are $o(1)$ as $p\to \infty$. Thus, at the end, \eqref{trace-avec-k>0-Eq} turns out to reduce to
\begin{multline}
\label{trace-avec-k>0-simple-Eq}
p^{-n}\int_{M} \psi_k(x)\tr_{\Lambda^{0,q}}[e^{-\frac{u}{p}\square_p}(x,x)]dv_M(x) = \\
 \frac{\mathrm{rk}(E)}{(2\pi)^{n}}\int_M \psi_k(x) \frac{\det(\dot{R}_{x}^{L})\tr_{\Lambda^{0,q}}[e^{u\omega_{d,x}}]}{\det\big(1-\exp(-u\dot{R}_{x}^{L})\big)}dv_M(x)+o(1).
\end{multline}

From \eqref{trace-comme-int-de-trace}, \eqref{trace-avec-k=0-Eq} and \eqref{trace-avec-k>0-simple-Eq}, we find that for $p\to \infty$,
\begin{equation}
\label{asymp-trace-Eq}
p^{-n}\tr_q[e^{-\frac{u}{p}\square_p}] =  \frac{\mathrm{rk}(E)}{(2\pi)^{n}}\int_M \frac{\det(\dot{R}_{x}^{L})\tr_{\Lambda^{0,q}}[e^{u\omega_{d,x}}]}{\det\big(1-\exp(-u\dot{R}_{x}^{L})\big)}dv_M(x)+o(1).
\end{equation}

On the other hand, for $x\in X$, let $\{\wt{w}_j\}_{1\leq j\leq n}$ be an local orthonormal frame  of $\wt{T^{1,0}M}_{U_x}$ such that $\dot{R}^L\wt{w}_j=a_j(\wt{Z})\wt{w}_j$ and let $\{\wt{w}^j\}_{1\leq j\leq n}$ be its dual basis. Then we have the following formula on $\wt{U}_x$:
\begin{equation}
\begin{aligned}
& \omega_d(\wt{Z}) = -\sum_{j=0}^n a_j(\wt{Z})\ol{\wt{w}}^j\wedge i_{\ol{\wt{w}}_j}, \\
& e^{u\omega_d(\wt{Z})} = \prod_{j=0}^n \big(1+(e^{-ua_j(\wt{Z})}-1)\big)\ol{\wt{w}}^j\wedge i_{\ol{\wt{w}}_j}, \\
& \frac{\det(\dot{R}_{x}^{L})\tr_{\Lambda^{0,q}}[e^{u\omega_{d,x}}]}{\det\big(1-\exp(-u\dot{R}_{x}^{L})\big)} = \bigg(\sum_{j_1<\dots <j_q}e^{-u \sum_{k=1}^q a_{j_k}(x)}\bigg)\prod_{j=0}^n \frac{a_j(x)}{\big( 1-e^{-ua_j(x)}\big)}.
\end{aligned}
\end{equation}
In particular, the term in the integral in the right-hand side of \eqref{asymp-trace-Eq} is uniformly bounded for $x\in M$ and $u>0$, and moreover,
\begin{equation}
\label{lim-u->infini}
\lim_{u\to \infty} \frac{1}{(2\pi)^{n}}\frac{\det(\dot{R}_{x}^{L})\tr_{\Lambda^{0,q}}[e^{u\omega_{d,x}}]}{\det\big(1-\exp(-u\dot{R}_{x}^{L})\big)}  = (-1)^q \boldsymbol{1}_{M(q)}(x)\det (\dot{R}^L_x/2\pi),
\end{equation}
where $\boldsymbol{1}_{M(q)}$ denotes the indicator function of $M(q)$. 

From Theorem  \ref{Euler-et-chaleur} and \eqref{asymp-trace-Eq}, we have for $0\leq q \leq n$ and any $u>0$
\begin{multline}
 \limsup_{p\to \infty} p^{-n} \sum_{j=0}^q (-1)^{q-j}\dim H^j(M,\l^p \otimes E) \\
 \leq \frac{\mathrm{rk}(E)}{(2\pi)^{n}}\int_M \sum_{j=0}^q\frac{\det(\dot{R}_{x}^{L})\tr_{\Lambda^{0,q}}[e^{u\omega_{d,x}}]}{\det\big(1-\exp(-u\dot{R}_{x}^{L})\big)}dv_M(x).
\end{multline}
This, together with \eqref{lim-u->infini} and dominated convergence for $u\to \infty$, gives
\begin{equation}
 \limsup_{p\to \infty} p^{-n} \sum_{j=0}^q (-1)^{q-j}\dim H^j(M,\l^p \otimes E) \leq (-1)^q \mathrm{rk}(E)\int_{M(\leq q)}\det (\dot{R}^L_x/2\pi) dv_M(x).
\end{equation}
Finally, we have
\begin{equation}
\det (\dot{R}^L_x/2\pi) dv_M(x) = \frac{1}{n!}\big( \frac{\ic}{2\pi}R^{L}\big)^{n},
\end{equation}
which conclude the proof of Theorem  \ref{OHMI}.
  \end{proof}

%%%%%%%%%%%%%%%%%%%%%%%%%%%%%%%%%%%%%%%%%%%%%%%%%%%%%%%%%%%%%%%%%%%%%%%%%%%%%%%%%%%%%%%%%%%%%%%%%%%%%%%%%%%%%%%%%%%%%%%%%%%%%%%%%%%%%%%%%%%%%%%%%%%%%%%%%%%%%%%%%%%%%%%%%%%%%%%%%%%%%%%%%%%%%%%%%%%%%%%%%%%%%%%%%%%%%%%%%%%%%%%%%%%%%%%%%%%%%%%%%%%%%%%%%%%%%%%%%%%%%%%%%%%%%%%%%%%%%%%%%%%%%%%
\section{Moishezon orbifolds}
\label{Sect-moishezon}

In this section, we introduce the concept of Moishezon orbifold, in a similar way as in the smooth case, and we give a criterion for a compact connected orbifold to be Moishezon. We thus prove that the Siu's \cite{MR755233,MR797421} and Demailly's \cite{demailly} answers to the Grauert-Riemenschneider conjecture \cite{MR0302938} are still valid in the orbifold case. We follow here the same lines as the presentation of these results given in \cite[Section 2.2]{ma-marinescu}.

In all this section, we consider a compact connected complex orbifold $M$, with set of orbifold charts $\mathcal{U}$. Also, as we mentioned before, we may assume without loss of generality that all the vector bundles in this section are proper.

%%%%%%%%%%%%%%%%%%%%%%%%%%%%%%%%%%%%%%%%%%%%%%%%%%%%%%%%%%%%%%%%%%%%%%%%%%%%%%%%%%%%%%%%%%%%%%%
\subsection{Definition of Moishezon orbifolds}
\label{sect-def-Moishezon}

Let $\mathscr{S}_M \subset \0_M$ be the subsheaf such that for every open $U$, $\mathscr{S}_M(U)$ consists of the functions $f\in \0_M(U)$ which do not vanish identically on any connected components of $U$.

\begin{definition}
 Let $\M$ be the sheaf associated with the pre-sheaf $U\mapsto \mathscr{S}_M(U)^{-1} \0_M(U)$. The section of $\M$ over an open set $U$ are called the \emph{meromorphic functions} on $U$.
\end{definition}

By definition, $f\in \M(U)$ can be written in a (connected) neighborhood $V\in \mathcal{U}$ of any point as $f=g/h$ with $g\in \0_M(V)$ and $h\in \mathscr{S}_M(V)$. On such a neighborhood, we thus have two $G_V$-invariant holomorphic functions $\tilde{g},\, \tilde{h}$ on $\wt{V}$, with $\tilde{h}\not\equiv 0$, such that the covering $\tilde{f}$ of $f$ is given by $\tilde{f}=\tilde{g}/\tilde{h}$. Note that it is \emph{a priori} stronger than asking that $f$ is covered by a $G_V$-invariant meromorphic function on $\wt{V}$. \\

\begin{definition}
We say that $f_1,\dots,f_k \in \M(M)$ are \emph{algebraically independent} if for any polynomial $P\in \C[z_1,\dots, z_k]$,  $P(f_1,\dots,f_k)=0$ implies $P=0$. 

 The \emph{transcendence degree} of $\M(M)$ over $\C$ is the maximal number of algebraically independent meromorphic functions on $M$. We denoted it by  $a(M)$.
\end{definition}

By a theorem of Cartan and Serre \cite{cartan} (see \cite{MR0084174}), we know that $(M, \0_M)$ is an complex space, and thus we have $a(M)\leq \dim M$.

\begin{definition}
 The compact connected orbifold $M$ is called a \emph{Moishezon orbifold} if it possesses $\dim M$ algebraically independent meromorphic functions, that is if $a(M)=\dim M$.
\end{definition}

%%%%%%%%%%%%%%%%%%%%%%%%%%%%%%%%%%%%%%%%%%%%%%%%%%%%%%%%%%%%%%%%%%%%%%%%%%%%%%%%%%%%%%%%%%%%%%%
\subsection{The Koidaira map and big line bundles}
\label{sect-koidaira-map}

Let $L$ be a Hermitian orbifold line bundle on $M$. Let $\mathbb{P}^*(H^0(M,L))$ be the Grassmannian of hyperplans of $H^0(M,L)$ (which can be identified to $\mathbb{P}(H^0(M,L)^*)$). 

As for smooth compact manifold, we define the \emph{base point locus} $\mathrm{BL}_{H^0(M,L)}$ of $H^0(M,L)$ as 
\begin{equation}
\mathrm{BL}_{H^0(M,L)}=\{x\in M \: : \: \forall \, s\in H^0(M,L), \, s(x)=0\},
\end{equation}
and we define the \emph{Kodaira map} by
\begin{equation}
\begin{aligned}
 \Phi_L \colon &M\setminus \mathrm{BL}_{H^0(M,L)} \to \mathbb{P}^*(H^0(M,L)) \\
  &x \mapsto \{ s\in H^0(M,L) \: : \:s(x)=0\}.
\end{aligned}
\end{equation}

If we chose a basis $(s_1,\dots,s_d)$ of $H^0(M,L)$, then we have the following description of $\Phi_L$: we use this basis to identify $H^0(M,L)\simeq \C^d$ and $\mathbb{P}^*(H^0(M,L))\simeq \mathbb{P}^{d-1}(\C)$, and under this identification we have
\begin{equation}
\label{eq-def-Kodaira-map}
\Phi_L (x) = [s_1(x),\dots, s_d(x)].
\end{equation}
This notation is slightly abusive because $s_i(x) \in L_x$. What is meant in \eqref{eq-def-Kodaira-map} is that if we chose a local basis $e_L$ of $L$ we can write $s_i(x)= f_i(x)e_L$ and then $\Phi_L (x) = [f_1(x),\dots, f_d(x)]$, but this element of $\mathbb{P}^{d-1}(\C)$ does not depend on $e_L$, which justifies the abuse of notation. In fact, on $\{s_i\neq 0\}$ we have $\Phi_L (x) = [s_1(x)/s_i(x),\dots, s_d(x)/s_i(x)]$.

In a local chart $(G_x,\wt{U}_x)$ near $x\in M$, the sections $s_i$ are covered by $G_x$-equivariant holomorphic sections $\tilde{s}_i \colon \wt{U}_x \to \wt{L}_{U_x}$ and the $G_x$-invariant cover $\wt{\Phi}_L\colon \wt{U}_x \to \mathbb{P}^{d-1}(\C)$ of $\Phi_L$ is given by $\wt{\Phi}_L (x) = [\tilde{s}_1(x),\dots, \tilde{s}_d(x)]$. From this, we see that $\Phi_L$ is an orbifold holomorphic function on $M$. \\

\begin{definition}
 For $p\in \N^*$, let $\varrho_p = \max \{ \mathrm{rk}_x \, \Phi_{L^p} \}_{x\in M_{reg}\setminus \mathrm{BL}_{H^0(M,L^p)}}$. 
 
 The \emph{Kodaira-Itaka} dimension $\kappa(L)$ of $L$ is $\kappa(L)= \max \{ \varrho_p \}_{p\in \N^*}$, and the line bundle $L$ is caled \emph{big} if $\kappa(L)=\dim M$.
\end{definition}

We have the following characterization of big line bundles.
\begin{theorem}
\label{thm-chara-big-line-bundle}
 Let $M$ be a compact connected complex orbifold of dimension $n$ and let $L$ be a holomorphic orbifold line bundle on $M$. Then $L$ is big if and only if
 \begin{equation}
\limsup_{p\to \infty} p^{-n} \dim H^0(X,L^p) >0.
\end{equation}
\end{theorem}

For the proof this theorem, we will need a version of Siegel's lemma \cite{MR0074061} (see also \cite[Lemma 2.2.6]{ma-marinescu}) for orbifolds. We will prove it first, drawing our inspiration from the approach of Andreotti~\cite{MR0152674}.

Let $x\in M$ and $r>0$ small. Let $(G_y,\wt{U}_y)$ be an orbifold chart as in Lemma \ref{Lem-carte-lineaire}, such that $x\in U_y$. For a pullback $\tilde{x}\in \wt{U}_y$ of $x$, we denote by $\wt{P}(\tilde{x},r)= \{ z\in \wt{U}_y \: : \: |z_i-\tilde{x}_i|\leq r \text{ for } 1\leq i\leq n\}$ the polydisc in $\wt{U}_y$ of center $\tilde{x}$ and of radius $r$ and  by $P(x,r)$ its image in $M$. The Shilov boundary $\wt{S}(\tilde{x},r)$ of $\wt{P}(\tilde{x},r)$ is defined by $\wt{S}(\tilde{x},r)=\{ z\in \wt{U}_x \: :\: |z_i-\tilde{x}_i|= r \text{ for } 1\leq i\leq n\}$, and its image in $M$ is denoted by $S(x,r)$. 

We warn the reader that the notations $P(x,r)$ and $S(x,r)$ can be somewhat misleading because these sets depend on the choice of $(G_y,\wt{U}_y)$ and of $\tilde{x}$, but it will not matter in the following proofs. 

In the sequel, if  $(G,\wt{U})$ is a chart and $X\subset \wt{U}$, we will set
\begin{equation}
\label{def-[A]}
[X] = \bigcup_{g\in G} g.X.
\end{equation}

\begin{lemma}
\label{lem-before-Siegel}
Let $M$ be a compact connected complex orbifold of dimension $n$ and let $L$ be a holomorphic orbifold line bundle on $M$. Consider points $x_1,\dots, x_m$ in $M_{reg}$ and positive numbers $r_1, \dots, r_m$. Using the above notations, we assume that $\wt{L}_{U_{y_i}}|_{\left[\wt{P}(\tilde{x}_i,2r_i)\right]}$ is (equivariantly) trivial for $1\leq i \leq m$ and $M\subset \cup_{i=1}^mP(x_i,e^{-1}r_i)$. Then there exists $k\in \N$ such that if $s\in H^0(M,L)$ vanishes at each $x_i$ up to order $k$, then $s=0$. In particular, $\dim H^0(M,L) \leq m \binom{n+k}{k}$.
\end{lemma}

\begin{proof}
 We first fix a trivialization of $\wt{L}_{U_{y_i}}$ over $\left[\wt{P}(\tilde{x}_i,2r_i)\right]$ for each $i$. 
 
 For $1\leq i,j\leq m$, let $K_{ij} := \ol{P(x_i,r_i)} \cap \ol{P(x_j,r_j)}$. We can cover $K_{ij}$ by finitely many disjoint orbifold charts  $(G_{V^\alpha_{ij}},\wt{V}^\alpha_{ij})$, $1\leq \alpha \leq A_{ij}$, with $V^\alpha_{ij}\subset U_{y_i}\cap U_{y_j}$. Let  $\wt{K}^\alpha_{ij}$ be the pre-image of $K_{ij}\cap V^\alpha_{ij}$ in $\wt{V}^\alpha_{ij}$. 
 
 For each $\alpha$, we choose an equivariant embedding $\varphi^\alpha_{ij\to i}\in \Phi_{V^\alpha_{ij},U_{y_i}}$ (resp. $\varphi^\alpha_{ij\to j}\in \Phi_{V^\alpha_{ij},U_{y_j}}$). Then the image of $\wt{K}^\alpha_{ij}$ is contained in $\left[\,\ol{\wt{P}(\tilde{x}_i,r_i)}\,\right]$ (resp. $\left[\,\ol{\wt{P}(\tilde{x}_j,r_j)}\,\right]$). This also induces a unique choice of compatible isomorphisms of $G_{V^\alpha_{ij}}$-equivariant bundles  $\varphi_{ij\to i}^{\alpha,L}\in \Phi_{V^\alpha_{ij},U_{y_i}}^L$ and $\varphi_{ij\to j}^{\alpha,L}\in \Phi_{V^\alpha_{ij},U_{y_j}}^L$.
 
 Then we can trivialize $\wt{L}_{V^{\alpha}_{ij}}$ either by composing the  isomorphism $\varphi_{ij\to i}^{\alpha,L}$ with the trivialization on $\left[\wt{P}(\tilde{x}_i,2r_i)\right]$, or by doing the same but replacing $i$ by $j$. This gives rise to $G_{V^\alpha_{ij}}$-invariant holomorphic transition functions:
 \begin{equation}
\tilde{c}^\alpha_{ij} \colon \wt{V}^\alpha_{ij} \to \C^*.
\end{equation}
Set 
\begin{equation}
C(L) = \sup \{|\tilde{c}^\alpha_{ij}(\tilde{x})|\}_{\tilde{x}\in \wt{K}^\alpha_{ij}, \, 1\leq \alpha \leq A_{ij}, \, 1\leq i,j \leq m }.
\end{equation}
As $\tilde{c}^{\alpha}_{ij}=(\tilde{c}^{\alpha}_{ji})^{-1}$, we have $C(L)\geq 1$. 

Set $k= \lfloor \log C(L)\rfloor +1 \in \N^*$, and consider $s\in H^0(M,L)$ vanishing at each $x_i$ up to order $k$. In the given trivialization, $s$ is covered for each $i$ by a $G_{y_i}$-equivariant function $\tilde{s}_i \colon \left[\wt{P}(x_i,2r_i)\right] \to \C$. Set
\begin{equation}
\| s\| = \sup \{ |\tilde{s}_i(\tilde{x})|\}_{\tilde{x}\in \left[\,\ol{\wt{P}(x_i,r_i)}\,\right],\, 1\leq i \leq m} = \sup \{ |\tilde{s}_i(\tilde{x})|\}_{\tilde{x}\in \ol{\wt{P}(x_i,r_i)},\, 1\leq i \leq m}.
\end{equation}
The last identity holds because $s_i$ is equivariant and a finite group acting linearly on $\C$ acts by isometries.

It is well-know that a holomorphic function $f$ on a neighborhood of a polydisc $P\subset \C^n$ attains its maximum on $\ol{P}$ at a point of the Shilov boundary of $P$. Thus, there exist $q\in \{1,\dots, m\}$ and $\tilde{t}_q\in \wt{S}(\tilde{x}_q,r_q)$ so that $|\tilde{s}_q(\tilde{t}_q)|=\|s\|$. We can find $j\neq q$ such that $t$, the image of $\tilde{t}_q$ in $M$, is in $P(x_j,e^{-1}r_j)$, and in particular $t\in K_{qj}$. 

Let $\alpha$ be such that $t\in K_{qj}\cap V^\alpha_{qj}$, and let $\tilde{t}^\alpha_{qj}$ be in the pre-image of $t$ in $\wt{V}^\alpha_{qj}$. Let $g\in G_{y_q}$ be such that $\tilde{t}_q=g\varphi^\alpha_{qj\to q}(\tilde{t}^\alpha_{qj})$. Set $\tilde{t}_j=\varphi^\alpha_{qj\to j}(\tilde{t}^\alpha_{qj})$, then, by the definition of $\tilde{c}^\alpha_{qj}$, we know that $\tilde{s}_q(g^{-1}\tilde{t}_q) = \tilde{c}^\alpha_{qj}(\tilde{t}^\alpha_{qj}) \tilde{s}_j(\tilde{t}_j)$. Hence,
\begin{equation}
\|s\| = |\tilde{s}_q(\tilde{t}_q)| = |\tilde{s}_q(g^{-1}\tilde{t}_q)| = |\tilde{c}^\alpha_{qj}(\tilde{t}^\alpha_{qj}) \tilde{s}_j(\tilde{t}_j)| \leq C(L)|\tilde{s}_j(\tilde{t}_j)|.
\end{equation}
Now, let $h\in G_{y_j}$ be such that $h\tilde{t}_j\in \wt{P}(\tilde{x}_j,e^{-1}r_j)$. Applying the Schwartz inequality (see for instance \cite[Problem 2.3]{ma-marinescu}) to $\tilde{s}_j$ on $\wt{P}(\tilde{x}_j,r_j)$ we get
\begin{equation}
|\tilde{s}_j(\tilde{t}_j)| =|\tilde{s}_j(h\tilde{t}_j)| \leq \|s\|. |h\tilde{t}_j|_0^k. r_j^{-k} \quad \text{where} \quad |(z_1,\dots,z_k)|_0= \sup\{|z_k|\}_{1\leq k \leq n}.
\end{equation}
Hence, as $|h\tilde{t}_j|_0\leq e^{-1}r_j$, we conclude:
\begin{equation}
\|s\| \leq \|s\| C(L)e^{-k},
\end{equation}
which implies that $s=0$ by the definition of $k$.

Finally, this proves that the map from $H^0(M,L)$ to the product of $m$ copies of the space $\C_k[X_1,\dots, X_n]$ of polynomial in $n$ variables and of degree $\leq k$, which associate to each section its $k$-jets at $x_1,\dots, x_m$, is injective. As $\dim \C_k[X_1,\dots, X_n] = \binom{n+k}{k}$, Lemma \ref{lem-before-Siegel} is proved.
  \end{proof}

\begin{theorem}[Siegel's lemma for orbifolds]
\label{thm-Siegel-orbifold}
  Let $M$ be a compact connected complex orbifold and let $L$ be a holomorphic orbifold line bundle on $M$. Then there exists $C>0$ such that for any $p\in \N^*$,
  \begin{equation}
\dim H^0(M,L^p) \leq C p^{\varrho_p}.
\end{equation}
\end{theorem}

\begin{proof}
 We modify the proof of Lemma \ref{lem-before-Siegel} and we will use the same notations as there. The set of points where $\Phi_{L^p}$ has rank less than $\varrho_p$ is a proper analytic set of $M$. As a consequence, the set $\{x \in M \: : \: \forall \, p\in \N^*, \, \mathrm{rk}_x \Phi_{L^p} = \varrho_p \}$ is dense in $M$ and we can choose $(x_1,\dots ,x_m)\in M_{reg}$  as in Lemma \ref{lem-before-Siegel} such that $\Phi_{L^p}$ has rank $\varrho_p$ at $x_i$ for any $p\in \N^*$. 

Since $\Phi_{L^p}$ has constant rank near $x_i$, there exists a $\varrho_p$-dimensional submanifold $M_{p,i}$ in a neighborhood of $x_i$ which is transversal at $x_j$ to the fiber $\Phi_{L^p}^{-1}(\Phi_{L^p}(x_i))$. Now, if $s\in H^0(M,L^p)$ vanishes up to order $k_p=p(\lfloor \log C(L)\rfloor +1)$ at each $x_i$ along $M_{p,i}$, then it also does along $\Phi_{L^p}^{-1}(\Phi_{L^p}(x_i))$, and thus it vanishes up to order $k_p$ at each $x_i$ (on $M$). With the same reasoning as in Lemma \ref{lem-before-Siegel} we get that $\|s\|\leq \|s\| C(L^p)e^{-k_p}$, and since $C(L^p)=C(L)^p$, we find that $s=0$.

Finally, as in Lemma \ref{lem-before-Siegel}, we have proved that the map associating to each $s\in H^0(M,L^p)$ its $k$-jets along $M_{p,i}$ for $1\leq i\leq m$ is injective. In particular, $\dim H^0(M,L^p) \leq m\binom{\varrho_p+k_p}{k_p}$, which implies Theorem \ref{thm-Siegel-orbifold}.
  \end{proof}

We can now turn back to the characterisation of big line bundles.

\begin{proof}[Proof of Theorem \ref{thm-chara-big-line-bundle}]
 Clearly, by Theorem \ref{thm-Siegel-orbifold}, we get that if $\displaystyle{\limsup_{p\to \infty}} \,p^{-n} \dim H^0(X,L^p) >0$ then $L$ is big.
 
 Conversely, if $L$ is big, there are $m\in \N^*$ and $x_0\in M_{reg}\setminus \mathrm{BL}_{H^0(M,L^m)} $ such that $\mathrm{rk}_{x_0}\Phi_{L^m} = n$. Thus, there are $s_0,\dots,s_n \in H^0(M,L^m)$ such that $s_0(x_0)\neq 0$ and $d(\frac{s_1}{s_0})_{x_0}\wedge \dots \wedge d(\frac{s_n}{s_0})_{x_0} \neq 0$. Hence, $(\frac{s_1}{s_0}(x), \dots , \frac{s_n}{s_0}(x))$ are local coordinates near $x_0$.
 
 Therefor, if a polynomial $P$ of degree $p$ in $n$ variables is such that $P(\frac{s_1}{s_0}, \dots , \frac{s_n}{s_0})=0$, then $P=0$. In particular, if $Q$ is a non-zero homogeneous polynomial of degree $p$ in $n+1$ variables, then $Q(s_0,\dots,s_n)\in H^0(M,L^{rm})$ is also non-zero. In deed if it was, then $P(X_1,\dots,X_n):=Q(1,X_1,\dots,X_n)\neq 0$ will satisfy
 $$P\left(\frac{s_1}{s_0}, \dots , \frac{s_n}{s_0}\right) = \frac{1}{s_0^r}Q(s_0,\dots,s_n)=0, $$
a contradiction. As the space of homogeneous polynomials of degree $p$ in $n+1$ variables has dimension $\binom{n+p}{p} \geq p^n/n!$, we deduce that $\dim H^0(M,L^{mp}) \geq p^n/n!$ and thus 
$$\displaystyle{\limsup_{p\to \infty}} \,p^{-n} \dim H^0(X,L^p) \geq 1/n! >0.$$
Theorem \ref{thm-chara-big-line-bundle} is proved.
  \end{proof}

\subsection{A criterion for Moishezon orbifolds}

\begin{lemma}
\label{lem-big-and-Moishezon}
  Let $M$ be a compact connected complex orbifold. If $M$ carries a big line bundle, then $M$ is Moishezon.
\end{lemma}

\begin{proof}
 If $L$ is a big line bundle on $M$, there are $m\in \N^*$ and $x_0\in M_{reg}\setminus \mathrm{BL}_{H^0(M,L^m)} $ such that $\mathrm{rk}_{x_0}\Phi_{L^m} = n$. Thus, there are $s_0,\dots,s_n \in H^0(M,L^m)$ such that $s_0(x_0)\neq 0$ and $d(\frac{s_1}{s_0})_{x_0}\wedge \dots \wedge d(\frac{s_n}{s_0})_{x_0} \neq 0$. Hence, $(\frac{s_1}{s_0}(x), \dots , \frac{s_n}{s_0}(x))$ are local coordinates near $x_0$. Therefor, if a polynomial $P$ in $n$ variables is such that $P(\frac{s_1}{s_0}, \dots , \frac{s_n}{s_0})=0$, then $P=0$. Thus, the meromorphic functions $\frac{s_1}{s_0}, \dots , \frac{s_n}{s_0}$ are algebraically independent, so that $a(M)\geq n$ and $M$ is Moishezon.
 \end{proof}

\begin{remark}
 In the case of a regular compact connected complex manifold, it is in fact equivalent to be Moishezon and to carry a big line bundle (see for instance \cite[Theorem 2.2.15]{ma-marinescu}), but in the singular case, the proof cannot be directly adapted and to the knowledge of the author it is not known wether it is also an equivalence.
\end{remark}

We can now prove the criterion stated in Theorem \ref{thm-criterion} in the introduction of this paper.
\begin{proof}[Proof of Theorem \ref{thm-criterion}]
 First, observe that if $(L,h^L)$ is semi-positive, then $X(1)=\emptyset$ and if moreover $(L,h^L)$ is positive at one point then $X(0)\neq \emptyset $ so that $\int_{M(\leq 1)} \left(\frac{\ic}{2\pi} R^L\right)^n = \int_{M( 0)} \left(\frac{\ic}{2\pi} R^L\right)^n$ is positive. Thus \emph{(i)} implies \emph{(ii)}. We will now prove that \emph{(ii)} implies that $M$ is Moishezon.
 
 If we apply Theorem \ref{OHMI} for $q=1$, we find
 \begin{equation}
\dim H^0(M,L^p) \geq \frac{p^n}{n!} \int_{M(\leq 1)} \left(\frac{\ic}{2\pi} R^L\right)^n + o(p^n),
\end{equation}
and thus with the hypothesis \eqref{eq-criterion} and Theorem \ref{thm-chara-big-line-bundle} we find that $L$ is big. By Lemma \ref{lem-big-and-Moishezon}, we find that $M$ is Moishezon. 
 \end{proof}

%%%%%%%%%%%%%%%%%%%%%%%%%%%%%%%%%%%%%%%%%%%%%%%%%%%%%%%%%%%%%%%%%%%%%%%%%%%%%%%%%%%%%%%%%%%%%%%%%%%%%%%%%%%%%%%%%%%%%%%%%%%%%%%%%%%%%%%%%%%%%%%%%%%%%%%%%%%%%%%%%%%%%%%%%%%%%%%%%%%%%%%%%%%%%%%%%%%%%%%%%%%%%%%%%%%%%%%%%%%%%%%%%%%%%%%%%%%%%%%%%%%%%%%%%%%%%%%%%%%%%%%%%%%%%%%%%%%%%%%%%%%%%%%

\bibliographystyle{spmpsci}      % mathematics and physical sciences

\begin{thebibliography}{10}
\providecommand{\url}[1]{{#1}}
\providecommand{\urlprefix}{URL }
\expandafter\ifx\csname urlstyle\endcsname\relax
  \providecommand{\doi}[1]{DOI~\discretionary{}{}{}#1}\else
  \providecommand{\doi}{DOI~\discretionary{}{}{}\begingroup
  \urlstyle{rm}\Url}\fi

\bibitem{MR0152674}
A.~Andreotti.
\newblock Th\'{e}or\`{e}mes de d\'{e}pendance alg\'{e}brique sur les espaces
  complexes pseudo-concaves.
\newblock {\em Bull. Soc. Math. France}, 91:1--38, 1963. 

\bibitem{MR852155}
Bismut, J.M.: The {W}itten complex and the degenerate {M}orse inequalities.
\newblock J. Differential Geom. \textbf{23}(3), 207--240 (1986).
\newblock \urlprefix\url{http://projecteuclid.org/euclid.jdg/1214440113}

\bibitem{MR886814}
Bismut, J.M.: Demailly's asymptotic {M}orse inequalities: a heat equation
  proof.
\newblock J. Funct. Anal. \textbf{72}(2), 263--278 (1987).
\newblock \doi{10.1016/0022-1236(87)90089-9}.
\newblock \urlprefix\url{http://dx.doi.org/10.1016/0022-1236(87)90089-9}

\bibitem{MR1060688}
Bismut, J.M.: Equivariant {B}ott-{C}hern currents and the {R}ay-{S}inger analytic
  torsion.
\newblock Math. Ann., \textbf{287}(3), 495--507, (1990).

\bibitem{bismut-lebeau}
Bismut, J.M., Lebeau, G.: Complex immersion and {Quillen} metrics.
\newblock Publ. Math. IHES \textbf{74}, 1--297 (1991)

\bibitem{MR1056777}
Bouche, T.: Convergence de la m{\'e}trique de {F}ubini-{S}tudy d'un fibr{\'e}
  lin{\'e}aire positif.
\newblock Ann. Inst. Fourier (Grenoble) \textbf{40}(1), 117--130 (1990).
\newblock \urlprefix\url{http://www.numdam.org/item?id=AIF_1990__40_1_117_0}

\bibitem{cartan}
H.~Cartan.
\newblock Quotient d'un espace analytique par un groupe discret
  d'automorphismes.
\newblock In {\em S\'{e}minaire Henri Cartan}, volume~6, pages 1--13,
  1953-1954.

\bibitem{MR0084174}
H.~Cartan.
\newblock Quotient d'un espace analytique par un groupe d'automorphismes.
\newblock In {\em Algebraic geometry and topology}, pages 90--102. Princeton
  University Press, Princeton, N. J., 1957.
\newblock A symposium in honor of S. Lefschetz.

\bibitem{MR2215454}
Dai, X., Liu, K., Ma, X.:
\newblock On the asymptotic expansion of {B}ergman kernel.
\newblock {\em J. Differential Geom.}, 72(1):1--41, 2006.

\bibitem{demailly}
Demailly, J.P.: Champs magn{\'e}tiques et in{\'e}galit{\'e}s de {Morse} pour la
  $d''$-cohomologie.
\newblock Ann. Inst. Fourier (Grenoble) \textbf{35}, 189--229 (1985)

\bibitem{MR1128538}
Demailly, J.P.: Holomorphic {M}orse inequalities.
\newblock In: Several complex variables and complex geometry, {P}art 2 ({S}anta
  {C}ruz, {CA}, 1989), \emph{Proc. Sympos. Pure Math.}, vol.~52, pp. 93--114.
  Amer. Math. Soc., Providence, RI (1991)

\bibitem{MR1603616}
Demailly, J.P.: {$L^2$} vanishing theorems for positive line bundles and
  adjunction theory.
\newblock In: Transcendental methods in algebraic geometry ({C}etraro, 1994),
  \emph{Lecture Notes in Math.}, vol. 1646, pp. 1--97. Springer, Berlin (1996).
\newblock \doi{10.1007/BFb0094302}.
\newblock \urlprefix\url{http://dx.doi.org/10.1007/BFb0094302}

\bibitem{MR2918158}
Demailly, J.P.: Holomorphic {M}orse inequalities and the
  {G}reen-{G}riffiths-{L}ang conjecture.
\newblock Pure Appl. Math. Q. \textbf{7}(4, Special Issue: In memory of Eckart
  Viehweg), 1165--1207 (2011).
\newblock \doi{10.4310/PAMQ.2011.v7.n4.a6}.
\newblock \urlprefix\url{http://dx.doi.org/10.4310/PAMQ.2011.v7.n4.a6}

\bibitem{MR0302938}
Grauert, H., Riemenschneider, O.: Verschwindungss{\"a}tze f{\"u}r analytische
  {K}ohomologiegruppen auf komplexen {R}{\"a}umen.
\newblock Invent. Math. \textbf{11}, 263--292 (1970)

\bibitem{MR3545500}
Hsiao, C.Y., Li, X.: Morse inequalities for {F}ourier components of
  {K}ohn-{R}ossi cohomology of {CR} manifolds with {$S^1$}-action.
\newblock Math. Z. \textbf{284}(1-2), 441--468 (2016).
\newblock \doi{10.1007/s00209-016-1661-6}.
\newblock \urlprefix\url{http://dx.doi.org/10.1007/s00209-016-1661-6}

\bibitem{MR2917156}
Hsiao, C.Y., Marinescu, G.: Szeg{\"o} kernel asymptotics and {M}orse
  inequalities on {CR} manifolds.
\newblock Math. Z. \textbf{271}(1-2), 509--553 (2012).
\newblock \doi{10.1007/s00209-011-0875-x}.
\newblock \urlprefix\url{http://dx.doi.org/10.1007/s00209-011-0875-x}

\bibitem{MR0474432}
Kawasaki, T.: The signature theorem for {$V$}-manifolds.
\newblock Topology \textbf{17}(1), 75--83 (1978)

\bibitem{MR641150}
Kawasaki, T.: The index of elliptic operators over {$V$}-manifolds.
\newblock Nagoya Math. J. \textbf{84}, 135--157 (1981).


\bibitem{MaTAMS}
Ma, X.: Orbifolds and analytic torsions.
\newblock Trans. Amer. Math. Soc. \textbf{357}, 2205--2233 (2005)

\bibitem{ma-marinescu}
Ma, X., Marinescu, G.: Holomorphic {M}orse inequalities and {B}ergman kernels,
  \emph{Progress in Mathematics}, vol. 254.
\newblock Birkh{\"a}user Verlag, Basel (2007)

\bibitem{MR2883416}
Ma, X., Marinescu, G.: Berezin-{T}oeplitz quantization and its kernel
  expansion.
\newblock In: Geometry and quantization, \emph{Trav. Math.}, vol.~19, pp.
  125--166. Univ. Luxemb., Luxembourg (2011)

\bibitem{MR0074061}
C.~L. Siegel.
\newblock Meromorphe {F}unktionen auf kompakten analytischen
  {M}annigfaltigkeiten.
\newblock {\em Nachr. Akad. Wiss. G{\"o}ttingen. Math.-Phys. Kl. IIa.},
  1955:71--77, 1955.
  
\bibitem{MR755233}
Siu, Y.T.: A vanishing theorem for semipositive line bundles over
  non-{K}{\"a}hler manifolds.
\newblock J. Differential Geom. \textbf{19}(2), 431--452 (1984).
\newblock \urlprefix\url{http://projecteuclid.org/euclid.jdg/1214438686}


\bibitem{MR797421}
Siu, Y.T.: Some recent results in complex manifold theory related to vanishing
  theorems for the semipositive case.
\newblock In: Workshop {B}onn 1984 ({B}onn, 1984), \emph{Lecture Notes in
  Math.}, vol. 1111, pp. 169--192. Springer, Berlin (1985).
\newblock \doi{10.1007/BFb0084590}.
\newblock \urlprefix\url{http://dx.doi.org/10.1007/BFb0084590}

\bibitem{MR1275204}
Siu, Y.T.: An effective {M}atsusaka big theorem.
\newblock Ann. Inst. Fourier (Grenoble) \textbf{43}(5), 1387--1405 (1993).
\newblock \urlprefix\url{http://www.numdam.org/item?id=AIF_1993__43_5_1387_0}

\bibitem{MR683171}
Witten, E.: Supersymmetry and {M}orse theory.
\newblock J. Differential Geom. \textbf{17}(4), 661--692 (1983) (1982).
\newblock \urlprefix\url{http://projecteuclid.org/euclid.jdg/1214437492}

\end{thebibliography}


\end{document}